\documentclass[11pt]{article} 
\usepackage{setspace}

\usepackage[utf8]{inputenc} 

\usepackage{subfigure}
\usepackage{caption}
\usepackage[margin=1in]{geometry}
\usepackage{graphicx} 
\usepackage{epstopdf}

\usepackage{booktabs} 
\usepackage{array} 
\usepackage{paralist} 
\usepackage{verbatim} 
\usepackage{subfig} 

\usepackage{fancyhdr} 
\usepackage[nottoc,notlof,notlot]{tocbibind} 
\usepackage[round]{natbib}

\usepackage[titles,subfigure]{tocloft} 

\usepackage{amsmath, amsthm, amssymb}

\usepackage{multirow}
\usepackage{algorithm}
\usepackage{algorithmic}
\usepackage{nicematrix}

\usepackage{amsthm}
\newtheorem{theorem}{Theorem}
\newtheorem{remark}{Remark}

\newtheorem{lemma}{Lemma}

\newtheorem{proposition}{Proposition}

\allowdisplaybreaks

\DeclareMathOperator*{\rank}{rank}
\DeclareMathOperator*{\conv}{conv}

\usepackage{rotating}
\usepackage{pdflscape}
\usepackage{pgfplots}
\usepackage{tikz}
\usetikzlibrary{shapes.geometric, arrows}
\usepgfplotslibrary{units}
\usepackage{url}
\usepackage{color}

\graphicspath{{figures/}} 

\allowdisplaybreaks

\usepackage{bm}

\newcommand{\cT}{\mathcal{T}}
\newcommand{\cS}{\mathcal{S}}
\newcommand{\cR}{\mathcal{R}}
\newcommand{\cQ}{\mathcal{Q}}

\pgfplotsset{compat=1.18}

\title{Improved Rank-One-Based Relaxations and \\ Bound Tightening Techniques for the Pooling Problem}
\author{Mosayeb Jalilian\thanks{mosayeb.jalilian@neoma-bs.fr, Faculty of Supply Chain and Operations Management,
NEOMA Business School, 76130 Mont-Saint-Aignan, France. This research was conducted when the first author was a master's student at Sabanc{\i} University.} \and Burak Kocuk\thanks{burakkocuk@sabanciuniv.edu, Industrial Engineering Program,  Sabanc{\i} University, Istanbul, Turkey 34956.}
}

\begin{document}

\maketitle

\begin{abstract}
The pooling problem is a classical NP-hard problem in the chemical process and petroleum industries. This problem is modeled as a nonlinear, nonconvex network flow problem in which raw materials with different specifications are blended in some intermediate tanks, and mixed again to obtain the final products with desired specifications. The analysis of the pooling problem is quite an active research area, and different exact formulations, relaxations and restrictions are proposed. In this paper, we focus on a recently proposed rank-one-based formulation of the pooling problem. In particular, we study a recurring substructure in this formulation defined by the set of nonnegative, rank-one matrices with bounded row sums, column sums, and the overall sum. We show that the convex hull of this set is second-order cone representable. In addition, we propose an improved compact-size polyhedral outer-approximation and families of valid inequalities for this set. We further strengthen these convexification approaches with the help of various bound tightening techniques specialized to the instances of the pooling problem. Our computational experiments show that the newly proposed polyhedral outer-approximation can improve upon the traditional linear programming relaxations of the pooling problem in terms of the dual bound. Furthermore, bound tightening techniques reduce the computational time spent on both the exact, linear programming and mixed-integer linear programming relaxations.
\end{abstract}

\noindent{\bf Keywords:}
    pooling problem; convexification; bound tightening; mixed-integer programming

\section{Introduction}\label{sec:introduction}
The classical blending problem that appears in many industrial settings involves determining the optimal blend of raw materials to produce a certain quantity of end products with minimum cost. This problem determines the proportions of the raw materials used in different products considering the incoming raw materials' specifications. The blending problem is polynomially solvable since it can be modeled as a compact-size linear program (LP). 

When the raw materials are blended in intermediate tanks and then mixed again to form the end products, the problem becomes considerably more challenging to solve due to its nonconvex nature. This problem is known as the \textit{pooling problem} and is one of the main problems in the chemical process and petroleum industries. The problem involves three types of tanks: inputs or sources to store raw materials, pools or intermediates to blend incoming flow streams and create new compositions, and outputs or terminals to store the final products.
There are two classes of pooling problems based on the links among the different tanks. The standard pooling problem has no flow stream among the pools, and the flow streams are source-to-terminal, source-to-pool, and pool-to-terminal. A typical {\it standard pooling problem} instance is shown in Figure~\ref{fig:typicalStandardPoolingProblem}. On the other hand, in the generalized pooling problem,  flow streams between the pools are allowed, which make  the problem even more challenging. This class was introduced by \cite{audet711pooling}, and an instance of a simple generalized pooling problem is shown in Figure~\ref{fig:typicalGeneralizedPoolingProblem}.

\begin{figure}[H]
\centering
\begin{minipage}{.5\textwidth}
  \begin{center}
        \begin{tikzpicture}[node distance=1.5cm, circle/.style={draw, circle, inner sep=0pt, minimum size=12mm}]
        \tikzstyle{input} = [draw, ellipse, minimum height=1cm, minimum width=1cm, fill=blue!20]
        \tikzstyle{output} = [draw, ellipse, minimum height=1cm, minimum width=1cm, fill=blue!20]
        \tikzstyle{pool} = [draw, ellipse, minimum height=1cm, minimum width=1cm, fill=blue!20]
        \node [input] (input1) at (0,0) {$s_{1}$};
        \node (input2) [input, below of=input1] {$s_{2}$};
        \node (input3) [input, below of=input2] {$s_{3}$};
        
        \node [pool] (pool1) at (3,-0.75) {$i_{1}$};
        \node (pool2) [pool, below of=pool1] {$i_{2}$};
        
        \node [output] (output1) at (6,-0.75) {$t_{1}$};
        \node (output2) [output, below of=output1] {$t_{2}$};
        
        \draw[->] (input1.east) -- (pool1.west);
        \draw[->] (input2.east) -- (pool1.west);
        \draw[->] (input3.east) -- (pool2.west);
        \draw[->] (input1.east) -- (pool2.west);
        
        \draw[->] (pool1.east) -- (output1.west);
        \draw[->] (pool1.east) -- (output2.west);
        \draw[->] (pool2.east) -- (output1.west);
        
        \end{tikzpicture}
    \end{center}
    \caption{A standard pooling problem instance.}
    \label{fig:typicalStandardPoolingProblem}
\end{minipage}%
\begin{minipage}{.5\textwidth}
  \begin{center}
        \begin{tikzpicture}[node distance=1.5cm, circle/.style={draw, circle, inner sep=0pt, minimum size=12mm}]
        \tikzstyle{input} = [draw, ellipse, minimum height=1cm, minimum width=1cm, fill=blue!20]
        \tikzstyle{output} = [draw, ellipse, minimum height=1cm, minimum width=1cm, fill=blue!20]
        \tikzstyle{pool} = [draw, ellipse, minimum height=1cm, minimum width=1cm, fill=blue!20]
        \node [input] (input1) at (0,0) {$s_{1}$};
        \node (input2) [input, below of=input1] {$s_{2}$};
        \node (input3) [input, below of=input2] {$s_{3}$};
        
        \node [pool] (pool1) at (3,-0.75) {$i_{1}$};
        \node (pool2) [pool, below of=pool1] {$i_{2}$};
        
        \node [output] (output1) at (6,-0.75) {$t_{1}$};
        \node (output2) [output, below of=output1] {$t_{2}$};
        
        \draw[->] (input1.east) -- (pool1.west);
        \draw[->] (input2.east) -- (pool1.west);
        \draw[->] (input3.east) -- (pool2.west);
        \draw[->] (input1.east) -- (pool2.west);
        
        \draw[->] (pool1.east) -- (output1.west);
        \draw[->] (pool1.east) -- (output2.west);
        \draw[->] (pool2.east) -- (output1.west);

        \draw[<->] (pool1.south) -- (pool2.north);
        
        \end{tikzpicture}
    
    \end{center}
    \caption{A generalized pooling problem instance.}
    \label{fig:typicalGeneralizedPoolingProblem}
\end{minipage}
\end{figure}

The analysis of the pooling problem has become an active research area since its introduction by \cite{haverly1978studies}. The nonconvexity of the problem arises due to keeping track of specifications throughout the network, leading to the potential existence of multiple local optima  \citep{alfaki2013multi}. Therefore, researchers have developed various exact formulations, relaxations, and heuristics to solve the problem.

The {\it P-formulation} was introduced by \cite{haverly1978studies}, which modeled the problem using flow and pool attribute quality variables. The authors used the Alternating Method to solve the problem recursively, where an LP model was generated using an estimation of the pool qualities. More recently, \cite{boland2015special} studied a problem consisting of multi-period variables which arises in the mining industry as a special case of the generalized pooling problem based on the {\it P-formulation}.

Another formulation, the {\it Q-formulation}, was proposed by \cite{ben1994global}, which used variables representing the relative proportions of pool input flows instead of the flow variables of pools in the {\it P-formulation}. The authors derived a general principle that can reduce or eliminate the duality gap of a nonconvex program and its Lagrangian dual in some special cases by partitioning the feasible set. They used this principle to compute a near-optimal solution that provides a primal bound for three different versions of the instance produced by \cite{haverly1978studies}.

Additionally, \cite{audet711pooling} proposed a hybrid formulation, which consists of the quality variable from the {\it P-formulation} and the proportion variable from the {\it Q-formulation} in addition to the flow variables. \cite{tawarmalani2002pooling} added some valid constraints to express mass balances across pools, creating the {\it PQ-formulation}. This formulation has proportion variables corresponding to sources and flow variables along the arcs between pools and terminals \citep{alfaki2013strong}. The authors relaxed the new constraints by the convex and concave envelopes \citep{al1983jointly} and proved the dominance of their results using convexification and disjunctive programming.


Furthermore, \cite{alfaki2013strong} introduced a model called the {\it TP-formulation} consisting of the proportion variables corresponding to the terminals and flow variables along with the arcs between sources and pools. They claimed that combining these two proportions (sources and terminals) leads to a new model referred to as the {\it STP-formulation} in which the full benefit is achieved. \cite{boland2016new} extended these approaches for the generalized pooling problem through \textit{source-based} and \textit{terminal-based} multi-commodity flow formulations. \cite{grothey2020effectiveness} introduced the \textit{QQ-formulation}, which only uses proportion variables to solve real-world instances in the animal feed mix industry.

It is worth mentioning that some studies have elaborated on the complexity of the pooling problem. While proving the NP-hardness of the pooling problem, \cite{alfaki2013multi} showed that the problem preserves NP-hardness even if there exists only one pool. \cite{baltean2018piecewise} demonstrated the strongly-polynomial solutions and the NP-hardness of the pooling problem by parameterizing the objective function concerning pool concentrations. 
Meanwhile, the problem remains NP-hard even by having only one quality constraint at each pool or when the number of sources and terminals are no more than two \citep{haugland2016computational}, but there exists a pseudo-polynomial algorithm to solve the problem \citep{haugland2016pooling}. On the other hand, the pooling problem could be polynomial-time solvable if there exists a bounded number of sources \citep{haugland2016pooling,boland2017polynomially}. Moreover, having only one source or one terminal makes the problem polynomially solvable since it can be formulated in the compact form as a linear program \citep{haugland2016computational}.

As can be seen above, the pooling problem is NP-Hard in general, and challenging to solve in practice. This has motivated the researchers to develop relaxations and restrictions for the problem to obtain dual and primal bounds.
The LP relaxations based on the McCormick envelopes \citep{mccormick1976computability} have been widely used in the literature to solve the pooling problem \citep{foulds1992bilinear,tawarmalani2002pooling,alfaki2013multi,boland2015special,dey2020convexifications}. In addition, mixed-integer programming (MIP) models have been developed to generate high-quality bounds as well \citep{adhya1999lagrangian,tomasgard2007optimization,pham2009convex,alfaki2011comparison,dey2015analysis,haugland2016pooling,gupte2017relaxations,gupte2019dynamic}. Furthermore, \cite{marandi2018numerical} conducted a numerical evaluation on the standard pooling problem instances  by applying the sum-of-squares hierarchy \citep{lasserre2017bounded}  via solving   semi-definite programs to construct lower bounds. Although this method has promising results in small instances, the scale of larger instances remains an issue and higher levels of the hierarchy  become computationally expensive. 

\cite{dey2020convexifications} proposed a new formulation for the pooling problem in which the bilinear constraints are replaced with rank-one constraints on the decomposed flow matrix variables related to a pool.  This allowed the authors to develop new relaxations for the pooling problem where the rank-one constraint with side constraints is relaxed. For example, they proved that the convex hull of the set of nonnegative, rank-one  matrices with bounded row (or column) sums and the overall sum is polyhedrally representable.  This translates to a nice interpretation for the pooling problem in which the bounds on row (resp. column) sums can be treated as the bounds on the incoming (resp. outgoing) arcs to a pool and the overall bound can be seen as the bound of the overall flow on the pool. 
We note that investigating the convex hulls of rank-one matrices carries significant implications for optimization, finding relevance in various domains such as machine learning, data analysis, and signal processing (see \cite{burer2017convexify,burer2009nonconvex,gupte2020extended,li2017convex,rahman2019facets,dey2020convexifications}).
In this paper, we also follow this approach and prove that the  convex hull of the set of nonnegative, rank-one  matrices with bounded row sums,  column sums, and the overall sum is second-order cone representable.  

While there have been some notable advancements in the field, most research has focused on the standard pooling problem and has been limited to small to medium problem instances. Moreover, the state-of-the-art solutions do not perform well while the flow streams among the pools are allowed. However, multi-period network flow problems, such as the mining problem associated with large-scale data, can be formulated as a special case of the general pooling problem. 

We aim to address the above challenges in our paper, 
which offers both theoretical and methodological contributions to the pooling problem literature. From the theoretical aspect, we prove that the  convex hull of the set of nonnegative, rank-one  matrices with bounded row sums,  column sums, and the overall sum is second-order cone representable. Although the size of this representation is exponential, it helps us to develop new strong LP relaxations than the well-known \textit{PQ}- and \textit{TP}-relaxations, and valid inequalities based on the reformulation and linearization technique (RLT). From the methodological aspect, we focus on improving both the time and quality of the exact and relaxed models via bound tightening.  To improve the lower and upper bounds on the capacities of the nodes and arcs, we use the optimization-based bound tightening (OBBT) technique. 
In addition, we develop a novel bound-tightening method that leverages the special structure of the mining instances, which are large-scale real-world problems that can be converted to generalized pooling problems. By implementing our method in a few simple steps, we can significantly improve the quality of the dual bounds and obtain the solution in a more reasonable time using the exact formulation.

The rest of this paper is organized as follows: In Section \ref{sec:socp}, we prove that the convex hull of the set of nonnegative, rank-one matrices with bounded row sums, column sums and overall sum is second-order cone representable. In addition, we  develop  polyhedral outer-approximations of this complicated convex hull and propose valid inequalities using RLT. These developments will be the basis of our analysis in the succeeding sections. 
In Section \ref{sec:formulations}, we describe the pooling problem formally and provide multi-commodity flow formulations in detail. Then, we review different polyhedral and MIP-based relaxations for these formulations. In Section \ref{sec:solution}, we present new LP relaxations we have developed. Moreover, we discuss the OBBT technique and how we can use it for the pooling problem. We also provide the details our  tailored bound-tightening method  to improve the bounds on the capacity of the arcs and nodes of the time-indexed pooling problem 
which arises in the mining industry. We present the settings of our different experiments and their computational results in Section \ref{sec:computations}. Finally, we will have some concluding remarks in Section \ref{sec:conclusion}. 
\\

\noindent{\bf Notation:} We denote the set of integers $1,\dots, n$ as $[n]$.
We will use the notation $\cdot$ when a bound is relaxed.
 $e_k$ is the $k$-th unit vector, $e$ is the vector of ones.
 
\section{Main Results} \label{sec:socp}



Let us define the following polyhedral set
\begin{equation}\label{set:convex}
\begin{split}
    \cT (l,u,l',u',L,U) := \bigg \{X \in \mathbb{R}_{+}^{m \times n}:  
     l_i \le \sum_{j=1}^n x_{ij} \le u_i, \ i \in[m] , \  
     l_j' \le  \sum_{i=1}^m x_{ij} \le u'_j, \ j \in [n] ,\\
     L \le \sum_{i=1}^m\sum_{j=1}^n  x_{ij} \le U  \bigg \},
\end{split}
\end{equation}
where $l,u \in \mathbb{R}_{+}^{m} $, $l',u' \in \mathbb{R}_{+}^{n} $ and $L,U \in \mathbb{R}_{+} $ such that $u \ge l$, $u' \ge l'$ and $U \ge L$. Without loss of generality, we assume that $u > 0$ and $u' > 0$ (otherwise, we can simplify the analysis by deleting the row $i$ with $u_i = 0$ and column $j$ with $u_j'=0$).  
In this section, we will focus on the study of the  nonconvex set
\begin{equation}\label{set:nonconvex Rank}
    \tilde \cT(l,u,l',u',L,U) := \left \{X \in  \cT(l,u,l',u',L,U) : \ \rank(X) \le 1  \right \}.
\end{equation}
This nonconvex set appears as a substructure in the pooling problem. As an illustration, consider pool $1$ in Figure~\ref{fig:typicalStandardPoolingProblem}. Let  $X \in \mathbb{R}_{+}^{2 \times 2} $ represent the decomposed flow variables, that is, $x_{ij}$ is the amount of flow originated at node $i$ and terminated at node $j$; $i \in\{s_1,s_2\}$, $j\in\{t_1,t_2\}$. In this case, the sum of row $i$ (resp. column $j$) entries of matrix $X$ is the incoming flow to (resp. outgoing flow from)  this pool from source $i$ (resp. to terminal $j$). Similarly, the   sum of overall entries of $X$ is the total flow at the pool. Due to the special structure of the pooling problem, we require $\rank(X) \le 1$, as this  guarantees that the outgoing flow from the pool will have identical specifications (see \cite{dey2020convexifications} for details).
 
\begin{remark}
The matrix $X$ has a slightly different interpretation in the case of generalized pooling problem. See Figure~\ref{fig:SamplePoolingRelaxations}   and related discussions.
\end{remark}

We study nonconvex set $\tilde \cT$ in three steps: In the first step, we prove that its convex hull is second-order cone representable in Section~\ref{sec:soc-repr}. This is an improvement over \cite{dey2020convexifications}, which showed that the convex hull of the set of nonnegative, rank-one matrices with bounded row  sums or column sums, and overall sum is polyhedrally representable (we call this set as \textit{column-wise relaxation} or \textit{row-wise relaxation}). Unfortunately, the size of our second-order cone  representation is exponential in the size of the matrix dimensions. This has motivated us to find a compact-size outer-approximation of the convex hull. In the second step, we obtain such a polyhedral outer-approximation in Section~\ref{sec:outer}, which is stronger than the intersection of column-wise and row-wise relaxations from \cite{dey2020convexifications}. In the third and final step, we use RLT to further strengthen the polyhedral outer-approximation obtained in the second step with the addition of valid inequalities in Section~\ref{sec:validIneq}.


\subsection{Second-Order Cone Representable Convex Hull}\label{sec:soc-repr}



In this section, we will  prove the following theorem:
\begin{theorem}\label{theorem:main}
$\conv(\tilde \cT(l,u,l',u',L,U) ) $ is second-order cone representable.
\end{theorem}

Before proving Theorem~\ref{theorem:main}, we remind the reader that  it has been recently proven that the set $\conv(\tilde \cT(l,u,l',u',L,U) ) $ has a polynomial-size polyhedral representation when either row bounds or column bounds  are relaxed \citep{dey2020convexifications}.
\begin{theorem}[\cite{dey2020convexifications}]\label{theorem:onlyRowOrColumn}
We have the following extended formulations:
\begin{itemize}
\item The row-wise formulation is 
$
\conv(\tilde \cT(l,u, \cdot,\cdot,L,U) ) =  \left\{ X \in \mathbb{R}_+^{m \times n}:  \exists t \in \mathbb{R}_+^n : \eqref{eq:extended xt} \right\},
$ where
\begin{equation}\label{eq:extended xt}
\begin{split}
&  l_i t_j \le  x_{ij} \le u_i t_j , \ i \in[m] , j \in[n] , \   L t_j \le \sum_{i=1}^m x_{ij} \le U t_j, \  j \in [n], \ \sum_{j=1}^n t_j = 1 .
\end{split}
\end{equation}
\item The column-wise formulation is 
$
\conv(\tilde \cT(\cdot,\cdot,l',u', L,U) ) =  \left\{ X \in \mathbb{R}_+^{m \times n}: \exists t' \in \mathbb{R}_+^m : \eqref{eq:extended xt'} \right\},
$ where
\begin{equation}\label{eq:extended xt'}
\begin{split}
&  l_j' t_i' \le  x_{ij} \le u_j' t_i' , \ i \in[m] , j \in[n] , \   L t_i' \le \sum_{j=1}^n x_{ij} \le U t_i', \  i \in [m], \ \sum_{i=1}^m t_i' = 1 .
\end{split}
\end{equation}
\end{itemize}
\end{theorem}

We need the following lemma in the proof of Theorem~\ref{theorem:main}.
\begin{lemma} \label{lemma:extrPt}Let $X$ be an extreme point of $ \tilde \cT (l,u,l',u',L,U)$.  Then,
\begin{itemize}
\item
$
\#_{row} := \big|\{ i \in [m] :  l_i < \sum_{j=1}^n x_{ij} < u_i \} \big| \le 1 .
$
\item
$
\#_{col} := \big |\{ j \in [n] :  l'_j < \sum_{i=1}^m x_{ij} < u'_j \} \big| \le 1 .
$
\end{itemize}
\end{lemma}
\begin{proof}
We only prove the first statement since the proof of the second statement is similar.
 Since $\rank(  X) \leq 1$ and $  X \ge 0$, there exist two non-zero vectors $  y \in \mathbb{R}_+^{n_1}$ and $  z \in \mathbb{R}_{+}^{n_2}$ such that $   X =   y   z^{\top}$. 
By contradiction, suppose that $\#_{row}>1$. Without loss of generality, let us assume that  $l_i < \sum_{j=1}^n x_{ij} < u_i$ for $i=1,2$, which implies that $y_1>0$ and $y_2 > 0$. Now, let us consider the following two points:
\[
X^{\pm} = y^\pm  z^\top \text{ where } y^\pm = y \pm \epsilon e_1 \mp \epsilon e_2.
\]
We have some observations: 
Firstly, the sum of the entries of $y$, $y^+$ and $y^-$ vectors is the same since $e^\top y = e^\top y^+ = e^\top y^-$.
Secondly, the row sums (except the first two) of  $X$, $X^+$, and $X^-$ matrices are the same since $e_i^\top (  y   z^{\top}) e = e_i^\top (  y^+   z^{\top}) e =e_i^\top (  y^-   z^{\top}) e $  for $i \ge 3$.
Thirdly, all the column sums   of  $X$, $X^+$ and $X^-$ matrices are the same since $e^\top (  y   z^{\top}) e_j = e^\top (  y^+   z^{\top}) e_j =e^\top (  y^-   z^{\top}) e_j $  for $j \ge 1$.
Fourthly, all the overall sums   of  $X$, $X^+$ and $X^-$ matrices are the same since $e^\top (  y   z^{\top}) e = e^\top (  y^+   z^{\top}) e =e^\top (  y^-   z^{\top}) e $.

Now, since all the row sums except the first two and all the column sums are unchanged, and the row sum bounds are not tight for the first two rows, we can find small enough $\epsilon > 0$ such that 
 both $X^+$ and $X^-$ belong to $ \tilde \cT (l,u,l',u',L,U)$ for some small enough $\epsilon>0$.
 Hence, since   $X = \frac12 X^+ + \frac12 X^-$ cannot be an extreme point,
 we reach a contradiction to the fact that $\#_{row}>1$.
\end{proof}

Now, we are finally ready to prove Theorem~\ref{theorem:main}.
\begin{proof}[Proof of Theorem~\ref{theorem:main}]
Since $ \tilde \cT (l,u,l',u',L,U)$ is compact,  $\conv(\tilde \cT (l,u,l',u',L,U))$ can be obtained as the convex hull of its extreme points. 

Let us consider a set, denoted by $\cS_{i,j}( \{b_{i'}\}_{i' \neq  i}, \{b'_{j'}\}_{j'\neq j})$, in which all the row bounds  and column bounds (except the $i$-th row and $j$-th column) are equal to $b_{i'}\in\{l_{i'}, u_{i'}\}$, $i'\neq i$ and $b_j'\in\{l'_{j'},u'_{j'}\}$, $j' \neq j$.
Assuming  $\cS_{i,j}( \{b_{i'}\}_{i' \neq  i}, \{b'_{j'}\}_{j'\neq j}) \neq\emptyset$, let $X $ belong to this set. Note that $X$ is an extreme point of $ \tilde \cT (l,u,l',u',L,U)$ due to Lemma~\ref{lemma:extrPt}. Since $\rank(  X) \leq 1$ and $  X \ge 0$, there exist two non-zero vectors $  y \in \mathbb{R}_+^{n_1}$ and $  z \in \mathbb{R}_{+}^{n_2}$ such that $   X =   y   z^{\top}$. 

Note that we have $y_{i'} \sum_{j=1}^n z_{j} = b_{i'}$, $i'\neq i$ and $z_{j'} \sum_{i=1}^m y_{i} = b'_{j'}$, $j'\neq j$. The rest of the proof involves considering three cases:

\noindent {\bf Case 1:} Suppose that there exist $I \neq i$ and $J \neq j$ such that $b_{I} > 0$ and $b'_{J}>0$. 
Then, we obtain the following relations:
\[
y_{i'} = \frac{b_{i'}}{b_I} y_I, \ i' \neq i \ \text{ and } \  z_{j'} = \frac{b'_{j'}}{b'_J} z_J, \ j' \neq j .
\]
As a shorthand notation, we define
\[
B := \sum_{i' \neq i} \frac{b_{i'}}{b_I} , \ \
B' := \sum_{j' \neq j} \frac{b'_{j'}}{b'_J} . 
\]
Considering the $i$-th row, $j$-th column and overall bounds, we obtain the following set of equations in $y$ and $z$,
\begin{equation*}
\begin{split}
&  y_I (z_j +B' z_J) = b_I, 
 \ z_J (y_i +B y_I) = b'_J \\
& y_i (z_j +B' z_J) \in [l_i,u_i],
\ z_j (y_i +B y_I) \in  [l'_j, u'_j], 
\ (y_i +B y_I) (z_j +B' z_J) \in  [L, U] \\
& y_i,y_I,z_j,z_J \ge 0,
\end{split}
\end{equation*}
which can be translated to $X$ variables as follows:
\begin{equation}\label{eq:substructure}
\begin{split}
& x_{ij}x_{IJ}=x_{iJ}x_{Ij} \\
& x_{Ij}+B' x_{IJ} = b_I, \ x_{iJ} +B x_{IJ} = b'_J \\
& x_{ij} +B' x_{iJ} \in [l_i,u_i], \ x_{ij} +B x_{Ij} \in  [l'_j, u'_j], \ x_{ij}+B'x_{iJ}+Bx_{Ij}+BB'x_{IJ} \in[L,U]
\\  
& x_{ij},x_{IJ},x_{iJ},x_{Ij}  \ge 0.
\end{split}
\end{equation}
The set defined by~\eqref{eq:substructure} is the intersection of a quadratic equation with a polytope, and its convex hull is known to be a second-order cone representable set \citep{santana2020}, which we denote by $\cQ_{ij}$. 
Hence, we conclude that the convex hull of $\cS_{i,j}( \{b_{i'}\}_{i' \neq  i}, \{b'_{j'}\}_{j'\neq j})$ is the following second-order cone representable set:
\begin{equation*}\label{eq:PolyC1}
\begin{split}
\left \{X \in \mathbb{R}_{+}^{m \times n}:   
  (x_{ij},x_{IJ},x_{iJ},x_{Ij}) \in \cQ_{ij}, 
  \sum_{j'=1}^n x_{ij}  = b_{i'} \ i'\in[m]\setminus\{i\}, \
  \sum_{i'=1}^m x_{i'j}  = b'_{j'} \  j' \in [n]\setminus\{j\}
 \right\}.
\end{split}
\end{equation*}

\medskip

\noindent {\bf Case 2:} Suppose that  $b_{i'} = 0$ for all $i' \neq i$, meaning that all rows of $X$ (except possibly for the $i$-th one) are  zero vectors. Then, we conclude that
$\cS_{i,j}( \{b_{i'}\}_{i' \neq  i}, \{b'_{j'}\}_{j'\neq j})$ is the following polyhedral set:
\begin{equation*}\label{eq:PolyR}
\begin{split}
\left \{X \in \mathbb{R}_{+}^{m \times n}:  
 l_i \le \sum_{j'=1}^n x_{ij'} \le u_i , \ 
 l'_{j'} \le x_{ij'} \le u'_{j'}  \ j'\in[n] , \
 x_{i'j'} = 0 \ i'\in[m]\setminus\{i\}, j' \in [n]
 \right\}.
\end{split}
\end{equation*}

\medskip

\noindent {\bf Case 3:} Suppose that $b'_{j'} = 0$ for all $j' \neq j$, meaning that all columns of $X$ (except possibly for the $j$-th one) are  zero vectors. Then, $\cS_{i,j}( \{b_{i'}\}_{i' \neq  i}, \{b'_{j'}\}_{j'\neq j})$ is the following polyhedral set:
\begin{equation*}\label{eq:PolyC2}
\begin{split}
\left \{X \in \mathbb{R}_{+}^{m \times n}:  
 l'_j \le \sum_{i'=1}^m x_{i'j} \le u'_j , \
 l_{i'} \le x_{i'j} \le u_{i'}  \  i'\in[m] , \
 x_{i'j'} = 0 \  i'\in [m], j' \in [n]\setminus\{j\}.
 \right\}.
\end{split}
\end{equation*}

In all cases, we conclude that the convex hull of  $\cS_{i,j}( \{b_{i'}\}_{i' \neq  i}, \{b'_{j'}\}_{j'\neq j}$ is second-order cone representable. 
Finally, by using the relation
\[
\conv(\tilde \cT (l,u,l',u',L,U)) = \conv \bigg ( \bigcup_{i\in[m], j\in[n]} \bigcup_{ b_{i'} \in \{l_{i'},u_{i'}\},b'_{j'} \in \{l'_{j'},u'_{j'}\}  } 
\text{conv}\big(  \cS_{i,j}( \{b_{i'}\}_{i' \neq  i}, \{b'_{j'}\}_{j'\neq j}) \big) \bigg),
\]
and utilizing the fact that the convex hull of the union of compact second-order cone representable sets is again second-order cone representable \citep{NemirovskiNotes}, we prove the statement of the theorem.
\end{proof}

\subsection{Polyhedral Outer-approximations}\label{sec:outer}

We proved that $\conv(\tilde \cT(l,u,l',u',L,U) ) $ is second-order cone representable in Theorem~\ref{theorem:main}. However, its exact representation might be quite large. Instead, we will develop some outer approximations of that set in this section.



\subsubsection{A Straightforward Polyhedral Outer-approximation}\label{sec:intersectionRowCol}

A straightforward outer-approximation can be obtained using  Theorem~\ref{theorem:onlyRowOrColumn} as follows: 
\begin{equation}\label{eq:intersectionRowCol}
    \cT^1(l,u,l',u',L,U) :=  \conv(\tilde \cT(l,u, \cdot,\cdot,L,U) ) \cap \conv(\tilde \cT(\cdot,\cdot,l',u', L,U) ) .
\end{equation}
Clearly, the following relation holds,
\[\conv(\tilde \cT(l,u,l',u',L,U) )  \subseteq \cT^1(l,u,l',u',L,U)\]
The set $ \cT^1(l,u,l',u',L,U) $ is the intersection of the row-wise and column-wise extended formulations derived in the previous section. Here, the variable $t_j$ ($t_i'$) represents the ratio of the column sum $j$ (row sum $i$) to the overall sum. We note that, due to the rank condition,  $t_j$ ($t_i'$) also represents the ratio of the entry $x_{ij}$ to the  row sum $i$ (column sum $j$).

We now present some  results related to the set $\cT^1(l,u,l',u',L,U) $. To help us with the illustration, let us define the set
\begin{equation*}
\begin{split}
\tilde \cT^1 (l,u,l',u',L,U) := \bigg \{ X \in \mathbb{R}_+^{m \times n}:  \exists t \in  \mathbb{R}_+^n, & t' \in \mathbb{R}_+^m : \\
& \eqref{eq:extended xt}-\eqref{eq:extended xt'},\\
& \rank(X) \le 1  , \\
& \ x_{ij}=t_j \sum_{j'=1}^n x_{ij'} =t_i' \sum_{i'=1}^m x_{i'j} , \ i \in[m] , j \in[n]  \bigg\}.
\end{split}
\end{equation*}

\begin{proposition} \label{prop: T T1}
$ \cT (l,u,l',u',L,U)  \supseteq   \cT^1 (l,u,l',u',L,U) $.
\end{proposition}
\begin{proof}
Let $X \in \cT^1 (l,u,l',u',L,U) $ and $t,t'$ satisfy the constraints in the description of equation \eqref{eq:intersectionRowCol}.
 
Summing each side of the inequality $ L t_j \le \sum_{i=1}^m x_{ij} \le U t_j $ over $j$ and using $\sum_{j=1}^n t_j = 1$ yield  $ L  \le \sum_{j=1}^n  \sum_{i=1}^m x_{ij} \le U $. 

Summing each side of the inequality $ l_i t_j \le x_{ij} \le u_i t_j $ over $j$ and using $\sum_{j=1}^n t_j = 1$ yield  $ l_i  \le \sum_{j=1}^n x_{ij} \le u_i $  for each $i\in[m]$.

Summing each side of the inequality $ L t_i' \le \sum_{j=1}^n x_{ij} \le U t_i' $ over $i$ and using $\sum_{i=1}^m t_i' = 1$ yield  $ L  \le  \sum_{i=1}^m \sum_{j=1}^n  x_{ij} \le U $. 

Summing each side of the inequality $ l_j' t_i' \le x_{ij} \le u_j' t_i' $ over $i$ and using $\sum_{i=1}^m t_i' = 1$ yield  $ l_j' \le  \sum_{i=1}^m   x_{ij} \le u_j' $ for each $j\in[n]$. 

Hence, we conclude that   $X \in  \cT (l,u,l',u',L,U)$.
\end{proof}

\begin{proposition}
$ \tilde \cT (l,u,l',u',L,U)  =  \tilde \cT^1 (l,u,l',u',L,U) $.
\end{proposition}
\begin{proof}
$ \tilde \cT (l,u,l',u',L,U) \supseteq   \tilde \cT^1 (l,u,l',u',L,U) $: It follows directly from Proposition \ref{prop: T T1}.

$ \tilde \cT (l,u,l',u',L,U)  \subseteq  \tilde \cT^1 (l,u,l',u',L,U) $: 
Let $X \in \tilde \cT (l,u,l',u',L,U)$. If $\rank(X)=0$, meaning that $X=0$, then we can simply set $t_1=1$ and $t'_1=1$. Since all the lower bounds have to be zero in this case (otherwise, $X=0$ would not have been feasible), it is trivial to see that $ X  \in \tilde \cT^1 (l,u,l',u',L,U)$.

On the other hand, if $\rank(X)=1$, then we set 
\[
t_j := \frac{x_{ij}}{\sum_{j'=1}^n x_{ij'}}  \quad \text{ and }\quad   t_i' := \frac{x_{ij}}{\sum_{i'=1}^m x_{i'j}} \quad \quad i \in[m], j \in[n].
\]
Observe that these definitions are well-defined since $\rank(X) = 1$. 
We trivially obtain $\sum_{j=1}^n t_j = 1$ and $\sum_{i=1}^m t_i' = 1 $. 

Multiplying each side of the inequality  $ L  \le \sum_{i=1}^m \sum_{j'=1}^n  x_{ij'} \le U $ by $t_j$ and replacing $\sum_{j'=1}^n  x_{ij'} t_j$ by $ x_{ij}$ yield $ L t_j \le \sum_{i=1}^m x_{ij} \le U t_j $ for each $j\in[n]$.

Multiplying each side of the inequality  $l_i  \le \sum_{j'=1}^n x_{ij'} \le u_i $ by $t_j$ and replacing $\sum_{j'=1}^n  x_{ij'} t_j$ by $ x_{ij}$ yield $ l_i t_j \le  x_{ij} \le u_i t_j $ for each $i\in[m]$, $j\in[n]$.

Multiplying each side of the inequality  $ L  \le  \sum_{j=1}^n  \sum_{i'=1}^m x_{i'j} \le U  $ by $t_i'$ and replacing $\sum_{i'=1}^m  x_{i'j} t_i'$ by $x_{ij}$ yield $ L t_i' \le \sum_{j=1}^n x_{ij} \le U t_i' $ for each $i\in[m]$.

Multiplying each side of the inequality  $ l_j' \le  \sum_{i'=1}^m   x_{i'j} \le u_j'  $ by $t_i'$ and replacing $\sum_{i'=1}^m  x_{ij'} t_i'$ by $ x_{ij}$ yield $ l_j' t_i' \le  x_{ij} \le u_j' t_i'  $ for each $i\in[m]$, $j\in[n]$.

Hence, we conclude that   $ X  \in \tilde \cT^1 (l,u,l',u',L,U)$.
\end{proof}

\subsubsection{A Stronger Polyhedral Outer-approximation}\label{sec:strongerPolyRelax}

Since the convex relaxation $ \cT^1 (l,u,l',u',L,U)$  is the intersection of the row-wise and column-wise extended formulations, it is natural to think of another extended formulation that considers rows and columns simultaneously. We now propose  a stronger polyhedral outer approximation. 

Let us define $ \cT^2 (l,u,l',u',L,U) := \{ X \in \mathbb{R}_{+}^{m \times n} : \exists R \in  \mathbb{R}_{+}^{m \times n}: \eqref{eq:extended xr} \}$ where
\begin{equation}\label{eq:extended xr}
\begin{split}
&  l_i  \sum_{i'=1}^m r_{i'j} \le  x_{ij} \le u_i  \sum_{i'=1}^m r_{i'j}  , \ i \in[m] , j \in[n] , \   L r_{ij} \le   x_{ij} \le U  r_{ij}  , \ i \in[m] , j \in[n] , \\ 
&  l_j' \sum_{j'=1}^n r_{ij'}  \le  x_{ij} \le u_j' \sum_{j'=1}^n r_{ij'}  , \ i \in[m] , j \in[n] , \   \sum_{i=1}^m \sum_{j=1}^n r_{ij} = 1   ,
\end{split}
\end{equation}
and
\begin{equation}\label{eq:extended xr + rank}
    \begin{split}
        \tilde \cT^2 (l,u,l',u',L,U)  := \bigg \{ X \in \mathbb{R}_{+}^{m \times n} : \exists R \in  \mathbb{R}_{+}^{m \times n}: \eqref{eq:extended xr} , &\ \rank(X) \le 1  ,\\
        &\ x_{ij} = r_{ij} \sum_{i'=1}^m \sum_{j'=1}^n x_{i'j'}  , \ i \in[m] , j \in[n]  \bigg \}
    \end{split}
\end{equation}

The variable $r_{ij}$   represents the ratio of the entry $x_{ij}$ to the overall sum. Intuitively, the relationships between the $r$ variables appeared in \eqref{eq:extended xr}, and the $t,t'$ variables appeared in \eqref{eq:extended xt}--\eqref{eq:extended xt'} are given as 
\[
r_{ij} = t_i't_j, \  t_i' = \sum_{j'=1}^n r_{ij'}  \text{ and }  t_j = \sum_{i'=1}^n r_{i'j} .
\]
Now, we will compare the relaxed extended formulations $T^1  (l,u,l',u',L,U) $ and $T^2  (l,u,l',u',L,U) $:
\begin{proposition} \label{prop: T1 T2}
$ \cT^1  (l,u,l',u',L,U) \supseteq \cT^2  (l,u,l',u',L,U) $.
\end{proposition}
\begin{proof}
Let $X \in  \cT^2 (l,u,l',u',L,U)$ and $R$ satisfy the constraints in equation \eqref{eq:extended xr}. Set 
\[
t_j := \sum_{i'=1}^m r_{i'j} \quad \text{ and }\quad  t_i' := \sum_{k'=1}^m r_{ij'} \quad \quad i \in[m], j \in[n].
\]
By construction, we have
\[\sum_{j=1}^n t_j = 1 \quad \text{and} \quad \sum_{i=1}^m t_i' = 1.\]
Also, 
\[l_i  \sum_{i'=1}^m r_{i'j} \le  x_{ij} \le u_i  \sum_{i'=1}^m r_{i'j} \implies l_i  t_j \le  x_{ij} \le u_i  t_j,\]
\[l_i' \sum_{k'=1}^m r_{ij'}  \le  x_{ij} \le u_i' \sum_{k'=1}^m r_{ij'} \implies l_i' t_i' \le  x_{ij} \le u_i' t_i'.\]  

Consider the inequality $ L r_{ij} \le   x_{ij} \le U  r_{ij}$;

Summing each side of it over $i$ and using $t_j = \sum_{i=1}^m r_{ij} $ yield $L t_j  \le   \sum_{i=1}^m x_{ij} \le U t_j, \quad \forall j \in [n].$

Summing each side of it over $j$ and using $ t_i' = \sum_{k=1}^m r_{ij}$ yield $L t_i'  \le   \sum_{j=1}^n x_{ij} \le U t_i', \quad \forall i \in [m].$

Hence, we conclude that   $X\in \cT^1 (l,u,l',u',L,U)$.
\end{proof}

\begin{proposition}
$ \tilde \cT (l,u,l',u',L,U)  =  \tilde \cT^2 (l,u,l',u',L,U) $.
\end{proposition}
\begin{proof}
$ \tilde \cT (l,u,l',u',L,U)  \supseteq  \tilde \cT^2 (l,u,l',u',L,U) $: It follows directly from Propositions \ref{prop: T T1} and \ref{prop: T1 T2}.

$ \tilde \cT  (l,u,l',u',L,U) \subseteq  \tilde \cT^2 (l,u,l',u',L,U) $:
Let $X \in \tilde \cT (l,u,l',u',L,U)$. {If $\rank(X)=0$, meaning that $X=0$, then we can simply set $ r_{11}=1$. Since all the lower bounds have to be zero in this case (otherwise, $X=0$ would not have been feasible), it is trivial to see that $ X  \in \tilde \cT^2 (l,u,l',u',L,U)$. On the other hand, if $\rank(X)=1$, then we} set 
\[
r_{ij} := \frac{x_{ij}}{\sum_{i'=1}^m \sum_{j'=1}^n x_{i'j'}}  \quad  \quad  i \in[m], j \in[n].
\]
We trivially obtain $\sum_{i=1}^m \sum_{j=1}^n r_{ij} = 1$. 
Since $\rank(X) = 1$, we also have
\[
\sum_{i'=1}^m r_{i'j} = \frac{\sum_{i'=1}^m x_{i'j} }{\sum_{j'=1}^n\sum_{i'=1}^m x_{i'j'}}  = \frac{x_{ij} }{\sum_{j'=1}^n x_{ij'}} \quad \text{ and } \quad
\sum_{j'=1}^n r_{ij'} = \frac{\sum_{j'=1}^n x_{ij'} }{\sum_{i'=1}^m \sum_{j'=1}^nx_{i'j'}}  = \frac{x_{ij} }{\sum_{i'=1}^m x_{i'j}},
\]
for each $i \in[m], j \in[n]$.

Multiplying each side of the inequality  $ L  \le \sum_{i'=1}^m \sum_{j'=1}^n  x_{i'j'} \le U $ by $r_{ij}$ and get $x_{ij}$ instead of $\sum_{i'=1}^m \sum_{j'=1}^n  x_{i'j'} r_{ij}$ yield $L r_{ij} \le x_{ij} \le U r_{ij}$ for each $i \in[m], j \in[n]$.

Multiplying each side of the inequality  $ l_i  \le \sum_{j'=1}^n  x_{ij'} \le u_i $ by $\sum_{i'=1}^m r_{i'j}$ and get $x_{ij}$ instead of $\sum_{j'=1}^n  x_{ij'}\sum_{i'=1}^m  r_{i'j}$ yield $l_i \sum_{i'=1}^m  r_{i'j} \le   x_{ij} \le u_i \sum_{i'=1}^m  r_{i'j} $ for each $i \in[m]$.

Multiplying each side of the inequality  $ l_j'  \le \sum_{j'=1}^n  x_{i'j} \le u_j' $ by $\sum_{j'=1}^n r_{ij'}$ and get $ x_{ij}$ instead of $\sum_{i'=1}^m  x_{i'j}\sum_{j'=1}^n  r_{ij'}$ yield $l_j' \sum_{i'=1}^m  r_{ij'} \le   x_{ij} \le u_j' \sum_{i'=1}^m  r_{ij'} $ for each $j \in[n]$.

Hence, we conclude that   $X \in \tilde \cT^2 (l,u,l',u',L,U)$.
\end{proof}

We conclude that both the sets $\cT^1 (l,u,l',u',L,U)$ and $\cT^2 (l,u,l',u',L,U) $ are outer approximations for $\conv(\cT (l,u,l',u',L,U))$, but $\cT^2 (l,u,l',u',L,U) $ yields a stronger relaxation than $\cT^1 (l,u,l',u',L,U) $. On the other hand, the extended formulation $\cT^1 (l,u,l',u',L,U)$ requires $m+n$ many additional variables while  we need  $mn$ many extra variables for $\cT^2 (l,u,l',u',L,U)$.

\subsection{Valid Inequalities Obtained by RLT}\label{sec:validIneq}

In this section, we strengthen  the polyhedral outer-approximation of $\conv(\tilde \cT(l,u,l',u',L,U) ) $ obtained in the previous section by using RLT. 
Assume that $L > 0$ and let us define the following set of inequalities
\begin{subequations}\label{eq:ext form t}
\begin{align}
        l'_j/ U \le \ &t_j \le  u'_j / L &    &  j \in[n]                      \label{eq:extt1} \\				
	l_i / U\le \ &t'_i \le  u_i / L  &  & i \in[m]                       \label{eq:extt2} \\					
	l_i t_j &\le u'_j t'_i           &    & i \in[m], j \in[n]          \label{eq:extt3} \\				
	l'_j t'_i &\le u_i t_j           &    & i \in[m], j \in[n]          \label{eq:extt4} ,
\end{align}			
\end{subequations}
and consider the set 
\begin{equation*}
\begin{split}
\tilde \cR (l,u,l',u',L,U)  := \bigg \{ X \in \mathbb{R}_{+}^{m \times n}: \exists t \in  \mathbb{R}_{+}^{n}, t' \in  \mathbb{R}_{+}^{m}, & R \in \mathbb{R}_{+}^{m \times n}: \\
& \eqref{eq:ext form t}, \\
& \sum_{j=1}^n t_{j} = 1, \\
& \sum_{i=1}^m t_{i}' = 1, \\
& \ r_{ij} = t_{i}'t_{j}, \ i \in[m] , j \in[n]  , \\
& \ x_{ij} = r_{ij} \sum_{i'=1}^m \sum_{j'=1}^n x_{i'j'} , \ i \in[m] , j \in[n]     \bigg \}.
\end{split}
\end{equation*}

\begin{proposition} \label{prop: RLT set}
We have $  \tilde \cR (l,u,l',u',L,U)  \supseteq \tilde \cT(l,u,l',u',L,U)  $.

\end{proposition}
\begin{proof}
Let $X \in \tilde\cT_{m,n}$. {If $\rank(X)=0$, meaning that $X=0$, then we can simply set $t_1=1$ and $t'_1=1$, and thus $r_{11}=1$. In this case, all the lower bounds have to be zero (otherwise, $X=0$ would not have been feasible), and it is trivial to see that $ X  \in \tilde \cR (l,u,l',u',L,U)$. On the other hand, if $\rank(X)=1$, then we set} $r_{ij} := \frac{x_{ij}}{\sum_{i'=1}^m \sum_{j'=1}^n x_{i'j'}}$ for $i \in [m]$ and for $j \in [n]$ ,  $t_i':= \sum_{j'=1}^n r_{ij'}$ for $i \in [m]$ and $t_j:= \sum_{i'=1}^m r_{i'j}$ for $j \in [n]$. We trivially obtain  $\sum_{j=1}^n t_j = 1$, {$ \sum_{i=1}^m t_i' = 1$},  $r_{ij} = t_i't_j$ and $ x_{ij} = r_{ij} \sum_{i'=1}^m \sum_{j'=1}^n x_{i'j'}$ for $ i \in[m]$, $ j \in[n]$.  

Dividing each side of the inequality $l_i \le \sum_{j=1}^n x_{ij} \le u_i$ by $\sum_{i'=1}^m \sum_{j'=1}^n x_{i'j'}$ and using the definition of $t_i'$ as above yield the inequality
\[
\frac{l_i}{\sum_{i'=1}^m \sum_{j'=1}^n x_{i'j'}} \le t_i' \le \frac{u_i}{\sum_{i'=1}^m \sum_{j'=1}^n x_{i'j'}} \implies 
{l_i}/{U} \le t_i' \le {u_i}/{L} , \quad i \in [m]. 
\]
Dividing each side of the inequality $l_j' \le \sum_{j=1}^n x_{ij} \le u_j'$ by $\sum_{i'=1}^m \sum_{j'=1}^n x_{i'j'}$ and using the definition of $t_j$ as above yield the inequality
\[
\frac{l_j'}{\sum_{i'=1}^m \sum_{j'=1}^n x_{i'j'}} \le t_j \le \frac{u_j'}{\sum_{i'=1}^m \sum_{j'=1}^n x_{i'j'}} \implies 
{l_j'}/{U} \le t_j \le {u_j'}/{L} , \quad j \in [n].
\]
The above derivation also shows that we have
\[
\frac{l_i}{t_i'} \le \sum_{i'=1}^m \sum_{j'=1}^n x_{i'j'} \le \frac{u_i}{t_i'}, \ i \in [m] \quad \text{and} \quad 
\frac{l_j'}{t_j} \le \sum_{i'=1}^m \sum_{j'=1}^n x_{i'j'} \le \frac{u_j'}{t_j}, \ j \in [n] ,
\] 
from which we deduce that $l_i t_j \le u_j' t_i'   $ and  $l_j' t_i' \le u_i t_j   $ for $ i \in[m]$, $ j \in[n]$. 

Hence, we prove that  $X \in \tilde\cR (l,u,l',u',L,U) $.
\end{proof}

We will obtain valid inequalities for the nonconvex set $\tilde \cR (l,u,l',u',L,U) $ using the RLT approach. Note that any such valid inequality is also valid for the set of our interest, $\tilde \cT (l,u,l',u',L,U)  $ due to Proposition  \ref{prop: RLT set}. We will apply the following procedure to obtain such inequalities:
\begin{enumerate}
\item Transform the inequalities \eqref{eq:ext form t} into the form  less-than-or-equal type with 0 right hand side.
\item Multiply the resulting inequalities to obtain bilinear expressions in $t$ and $t'$, and convert them into inequalities in $r$:
\begin{enumerate}
\item Replace the term $t_i' t_j$ with $r_{ij}$.
\item Replace the term $t_i'$ with $\sum_{j'=1}^n r_{ij'}$.
\item Replace the term $t_j$ with $\sum_{i'=1}^m r_{i'j}$.
\end{enumerate}
\item Obtain inequalities in $x$ variables by multiplying $\sum_{i'=1}^m \sum_{j=1}^n x_{i'j'}$ with the inequalities obtain in the previous step in $r$ variables:
\begin{enumerate}
    \item Replace the term $r_{ij} \sum_{i'=1}^m \sum_{j'=1}^n x_{i'j'}$ with $x_{ij}$.
    \item Replace the term $\sum_{i'=1}^m r_{i'j} \sum_{i'=1}^m \sum_{j'=1}^n x_{i'j'}$ with $\sum_{i'=1}^m x_{i'j}$.
    \item Replace the term $\sum_{j'=1}^n r_{ij'} \sum_{i'=1}^m \sum_{j'=1}^n x_{i'j'}$ with $\sum_{j'=1}^n x_{ij'}$.
\end{enumerate}
\end{enumerate}

\noindent{\bf Multiplying \eqref{eq:extt1} and \eqref{eq:extt2}}: These are  precisely the McCormick envelopes applied to 
$r_{ij} = t_i' t_j$. These linear inequalities in $r$ variables are given as follows:
\begin{equation*}
\begin{split}
r_ {ij} &\ge (l_i/U) \sum_{i'=1}^m r_{i'j} + (l_j'/U) \sum_{j'=1}^n r_{ij'} - (l_i l_j')/(U^2) \\ 
r_ {ij} &\le (l_i/U) \sum_{i'=1}^m r_{i'j} + (u_j'/L) \sum_{j'=1}^n r_{ij'} - (l_i u_j')/(UL) \\ 
r_ {ij} &\le (u_i/L) \sum_{i'=1}^m r_{i'j} + (l_j'/U) \sum_{j'=1}^n r_{ij'} - (u_i l_j')/(UL) \\ 
r_ {ij} &\ge (u_i/L) \sum_{i'=1}^m r_{i'j} + (u_j'/L) \sum_{j'=1}^n r_{ij'} - (u_i u_j')/(L^2)
\end{split}
\end{equation*}
The linear inequalities in $x$ variables are obtained as follows:
\begin{equation*}
\begin{split}
x_{ij} &\ge (l_i/U) \sum_{i'=1}^m x_{i'j} + (l_j'/U) \sum_{j'=1}^n x_{ij'} - (l_i l_j')/(U^2) \sum_{i'=1}^m\sum_{j'=1}^n x_{i'j'}\\ 
x_{ij} &\le (l_i/U) \sum_{i'=1}^m x_{i'j} + (u_j'/L) \sum_{j'=1}^n x_{ij'} - (l_i u_j')/(UL)\sum_{i'=1}^m\sum_{j'=1}^n x_{i'j'} \\ 
x_{ij} &\le (u_i/L) \sum_{i'=1}^m x_{i'j} + (l_j'/U) \sum_{j'=1}^n x_{ij'} - (u_i l_j')/(UL)\sum_{i'=1}^m\sum_{j'=1}^n x_{i'j'} \\ 
x_{ij} &\ge (u_i/L) \sum_{i'=1}^m x_{i'j} + (u_j'/L) \sum_{j'=1}^n x_{ij'} - (u_i u_j')/(L^2)\sum_{i'=1}^m\sum_{j'=1}^n x_{i'j'}
\end{split}
\end{equation*}

\noindent{\bf Multiplying \eqref{eq:extt3} and \eqref{eq:extt4} for the same $(i,j)$ pair}: 
The second-order cone representable inequalities in $r$ variables are given as follows:
\begin{equation*}
\begin{split}
l_i u_i \bigg(\sum_{i'=1}^m r_{i'j} \bigg)^2 + l_j' u_j' \bigg(\sum_{j'=1}^n r_{ij'}\bigg)^2 \le (l_i l_j' + u_i u_j')r_{ij}
\end{split}
\end{equation*}
The inequalities in $x$ variables, which are again second-order cone representable, are obtained as follows:
\begin{equation*}
\begin{split}
l_i u_i \bigg(\sum_{i'=1}^m x_{i'j} \bigg)^2 + l_j' u_j' \bigg(\sum_{j'=1}^n x_{ij'}\bigg)^2 \le (l_i l_j' + u_i u_j')x_{ij}  \sum_{i'=1}^m\sum_{j'=1}^n x_{i'j'}
\end{split}
\end{equation*}

\noindent{\bf Multiplying \eqref{eq:extt1} and \eqref{eq:extt3} for the same $(i,j)$ pair}: We obtain two types of inequalities. 
{The first type is second-order cone representable. These inequalities in $r$ variables are as follows:} 
\begin{equation*}
\begin{split}
l_i \bigg(\sum_{i'=1}^m r_{i'j} \bigg)^2  \le (l_i l_j'/U) \sum_{i'=1}^m r_{i'j} - (l_j' u_j' /U) \sum_{j'=1}^n r_{ij'} + u_j' r_{ij}.
\end{split}
\end{equation*}
{
The inequalities in $x$ variables are obtained as follows:
\begin{equation*}
    l_{i} \bigg(\sum_{i'=1}^m x_{i'j} \bigg)^2  \le \bigg[ (l_{i} l'_{j}/U) \sum_{i'=1}^m x_{i'j} - (l'_{j} u'_{j} /U) \sum_{j'=1}^n x_{ij'} + u'_{j} x_{ij}\bigg] \sum_{i'=1}^m \sum_{j'=1}^n x_{i'j'}
\end{equation*}
}

The second type is a reverse-convex inequality
\begin{equation*}
\begin{split}
u_i   t_j ^2  \ge (u_i l_j'/U) \sum_{i'=1}^m r_{i'j} - (l_j'^2 /U) \sum_{j'=1}^n r_{ij'} + l_j' r_{ij}.
\end{split}
\end{equation*}
We can over-approximate the left-hand side to obtain the following valid linear inequality in $r$ variables:
\begin{equation}\label{eq:validac_r}
u_i \bigg[ (u_j'/L+l_j'/U)  \sum_{i'=1}^m r_{i'j} - (u_j'l_j')/(UL)  \bigg] \ge (u_i l_j'/U) \sum_{i'=1}^m r_{i'j} - (l_j'^2 /U) \sum_{j'=1}^n r_{ij'} + l_j' r_{ij}.
\end{equation}
The corresponding inequalities in $x$ variables are obtained as follows:
\begin{equation}\label{eq:validac_x}
\begin{split}
u_i \bigg[ (u_j'/L+l_j'/U)  \sum_{i'=1}^m x_{i'j} - (u_j'l_j')/(UL)  \sum_{i'=1}^m\sum_{j'=1}^n x_{i'j'}  \bigg] & \ge \\
(u_i l_j'/U) \sum_{i'=1}^m x_{i'j} - (l_j'^2 /U) \sum_{j'=1}^n x_{ij'} + l_j' x_{ij}.
\end{split}
\end{equation}
We note that similar valid inequalities can be obtained if the inequalities  \eqref{eq:extt1} and \eqref{eq:extt4},   \eqref{eq:extt2} and \eqref{eq:extt4}, or   \eqref{eq:extt2} and \eqref{eq:extt4} are multiplied.

\section{The Pooling Problem}\label{sec:formulations}
In this section, we are going to describe the pooling problem formally, and present its the well-known and recently developed exact formulations and different types of relaxations.

\subsection{Problem Definition and Notation}
Let $G=(N,A)$ represent a graph with the node set $N$ and the arc set $A$. Moreover, let $S$, $I$, $T$, and $K$ denote the set of sources (inputs), intermediates (pools), terminals (outputs), and specifications respectively. Then, in the pooling problem, we have $N = S \cup I \cup T$. For the standard pooling problem , we have $A \subseteq (S \times (I \cup T)) \cup (I \cup T)$, while in the general pooling problem, we have $A \subseteq (S \times (I \cup T)) \cup (I \times (I \cup T))$. This definition clearly shows that in the generalized version, we may have flow streams among the pools. In this notation, $S_i$ is the set of source nodes from which there is a path to node $i$ and $T_i$ is the set of terminal nodes to which there is a path from node $i$. The set of nodes to which there is an arc from node $i$ and the set of nodes from which there is an arc to node $i$ are denoted by $N_i^+$ and $N_i^-$ respectively. In this notation, the unit cost of using arc $(i,j)$ is shown by $C_{ij}$ and the specification $k$ of source $s$ by $\lambda_k^s$. We may also have some lower and upper bounds for the desired specification $k$ at terminal $t$ denoted by $[\underline{\mu}_k^t,\bar{\mu}_k^t]$, the capacity of node $i$ denoted by $[L_i,U_i]$, and capacity of arc $(i,j)$ denoted by $[l_{ij},u_{ij}]$. A summary of all the notation, which we have mostly adapted from \cite{dey2020convexifications}, can be found in Table \ref{tab:sbnotation}.

\begin{table}[H]
  \centering
  \caption{Notations of the source-based rank formulation}
  \bigskip
    \scalebox{1}{
    \begin{tabular}{ccl}
    \toprule
    Indices & $s$     & source (or input), $s=1,…,S$ \\
          & $i$     & intermediate (or pool), $i=1,…,I$ \\
          & $t$     & terminal (or output), $t=1,…,T$ \\
          & $k$     & specification, $k=1,…,K$ \\
    Sets  & $S_i$   & the set of source nodes from which there is a path to node $i$\\
          & $T_i$   & the set of terminal nodes to which there is a path from node $i$\\
          & $N_i^+$ & the set of nodes to which there is an arc from node $i$\\
          & $N_i^-$ & the set of nodes from which there is an arc from node $i$\\
    Variables & $f_{ij}$ & the amount of flow from node $i$ to node $j$ \\
          & $x_{ij}^s$ & the amount of flow on arc $(i,j)$ originated at the source $s \in S_i$ \\
          & $q_{i}^{s}$ & the fraction of flow at pool $i$ originated at source $s$ \\
    Parameters & $C_{ij}$  &   cost of sending unit flow over arc $(i,j)$ \\
          & $\lambda_k^s$  & the specification $k$ of source $s$ \\
          & $[\underline{\mu}_k^t,\bar{\mu}_k^t]$  & the desired interval for specification $k$ of terminal $t$ \\
          & $[L_i,U_i]$ & lower bound and upper bound of the capacity of node $i$  \\
          & $[l_{ij},u_{ij}]$ & lower bound and upper bound of the capacity of arc $(i,j)$ \\
    \bottomrule
    \end{tabular}}%
  \label{tab:sbnotation}%
\end{table}%

\subsection{Source-Based Rank Formulation}
In this section, we review the source-based multi-commodity flow formulation for the generalized pooling problem developed in \cite{alfaki2013multi}. This formulation consists of the proportion variables corresponding to sources and the flow variables along with the arcs between pools and terminals. 
This formulation was later investigated and presented in \cite{dey2020convexifications}. 
These authors have convexified the nonconvex constraint in different ways. We will use their formulation and go over their proposed methods in the next sections.

\subsubsection{Mathematical Model}
In this section, we will review the {\it Source-Based} multi-commodity flow formulation. An explanation of the mathematical model and an introduction to different relaxations and restrictions will follow.  
\begin{align}
    \min & \sum_{i \in I}\sum_{s \in S_i}\sum_{j \in N_i^+} C_{si}x_{ij}^s-\sum_{i \in I}\sum_{j \in N_i^+} C_{ij}f_{ij} \label{sb:objective}\\
    \text{ s.t. } & L_i \le \sum_{j \in N_i^-} f_{ji} \le U_i & \forall i \in I \cup T \label{sb:pools&terminalscapacity}\\
    &L_i \le \sum_{j \in N_i^+} f_{ij} \le U_i & \forall i \in S \label{sb:sourcecapacity}\\
    &l_{ij} \le f_{ij} \le u_{ij} & \forall (i,j) \in A \cup \{(s,i): i \in I, s \in S_i\} \label{sb:arccapacity}\\
    &\sum_{j \in N_{si}^-} x_{ji}^s=\sum_{j \in N_i^+} x_{ij}^s & \forall i \in I, \forall s\in S_i \label{sbbalance}\\
    &\sum_{s\in S_i}x_{ij}^s=f_{ij} & \forall (i,j) \in A \label{sb:flowdecomposition}\\
    &\sum_{j \in N_i^+} x_{ij}^s=f_{si} & \forall i \in I, \forall s \in S_i \label{sb:ghostflow}\\
    &\underline{\mu}_k^t \sum_{j \in N_t^-} f_{jt} \le \sum_{j \in N_t^-} \sum_{s \in S_j} \lambda_k^s x_{jt}^s \le \bar{\mu}_k^t\sum_{j \in N_t^-} f_{jt} & \forall t \in T, \forall k \in K \label{sb:specifications}\\
    & x_{ij}^s=q_i^s f_{ij} & \forall (i,j) \in A, \forall s \in S_i \label{sbbilinear}\\
    & x_{ij}^s \ge 0 & \forall (i,j) \in A, \forall s \in S_i \label{sb:x_nonnegativity}\\
    & q_i^s \ge 0 & \forall i \in I, \forall s \in S_i. \label{sb:q_nonnegativity}
\end{align}

The objective function \eqref{sb:objective} minimizes the cost of sending the raw materials to the outputs through some pools. In this equation, the sum of $x_{ij}^s$ variables can be replaced by $f_{ij}$ and the objective will be as follows:
\[\min \sum_{(i,j) \in A} C_{ij}f_{ij}\]
In this case, we should consider the cost of purchasing raw materials as positive and the profit of selling outputs as negative. In addition, the cost of sending the raw material directly to the output is the difference between its cost and the revenue from selling it which could be either positive or negative.

Constraint \eqref{sb:pools&terminalscapacity} imposes bounds on the capacity of pools and terminals while in constraint \eqref{sb:sourcecapacity}, we have these bounds for the sources. In constraint \eqref{sb:arccapacity}, flows on different arcs are limited to be in an interval, and constraint \eqref{sbbalance} is the flow conservation, which guarantees that all the flows coming into pools go out of them. We define the set $N_{si}^-$ used in equation (\ref{sbbalance}) as
\[N_{si}^- = \{ j \in N_i^- : j \notin S\setminus s \text{ and } s \in S_j \}.\]

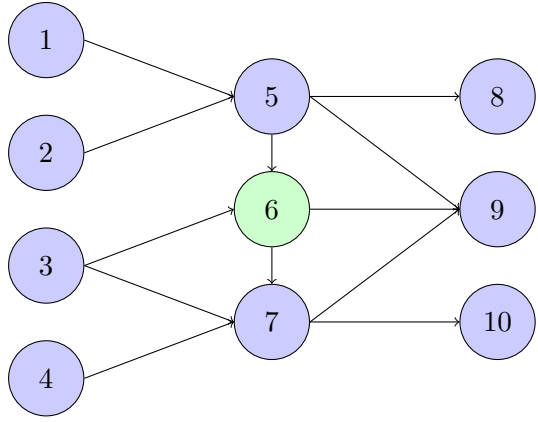
\begin{figure}[H]
    \begin{center}
        \begin{tikzpicture}[node distance=1.5cm, circle/.style={draw, circle, inner sep=0pt, minimum size=12mm}]
        \tikzstyle{input} = [draw, ellipse, minimum height=1cm, minimum width=1cm, fill=blue!20]
        \tikzstyle{output} = [draw, ellipse, minimum height=1cm, minimum width=1cm, fill=blue!20]
        \tikzstyle{pool} = [draw, ellipse, minimum height=1cm, minimum width=1cm, fill=blue!20]
        \node [input] (input1) at (0,0) {$1$};
        \node (input2) [input, below of=input1] {$2$};
        \node (input3) [input, below of=input2] {$3$};
        \node (input4) [input, below of=input3] {$4$};
        
        \node [pool] (pool1) at (3,-0.75) {$5$};
        \node [pool, fill=green!20] (pool2) at (3,-2.25) {$6$};
        \node (pool3) [pool, below of=pool2] {$7$};
    
        \node [output] (output1) at (6,-0.75) {$8$};
        \node [output] (output2) at (6,-2.25) {$9$};
        \node (output3) [output, below of=output2] {$10$};
        
        \draw[->] (input1.east) -- (pool1.west);
        \draw[->] (input2.east) -- (pool1.west);
        \draw[->] (input3.east) -- (pool2.west);
        \draw[->] (input3.east) -- (pool3.west);
        \draw[->] (input4.east) -- (pool3.west);

        \draw[->] (pool1.south) -- (pool2.north);
        \draw[->] (pool2.south) -- (pool3.north);
        
        \draw[->] (pool1.east) -- (output1.west);
        \draw[->] (pool1.east) -- (output2.west);
        \draw[->] (pool2.east) -- (output2.west);
        \draw[->] (pool3.east) -- (output2.west);
        \draw[->] (pool3.east) -- (output3.west);
        
        \end{tikzpicture}
    \end{center}
    \caption{A sample generalized pooling problem instance.}
    \label{fig:SamplePoolingRelaxations}
\end{figure}

Figure~\ref{fig:SamplePoolingRelaxations} shows a fictitious sample of a general pooling problem, which will be our running example in this section. Let us write constraint \eqref{sbbalance} for pool $i=6$ and its source $s=1$: 
\[N_{1,6}^-=\{5\},~N_6^+=\{7,9\} \implies x_{5,6}^1=x_{5,7}^1+x_{5,9}^1\]
In fact, this constraint ensures that the incoming flow to a pool from each of its source nodes equals the outgoing flow from it originated at the same source.

Equations \eqref{sb:flowdecomposition} and \eqref{sb:ghostflow} ensure the flow decomposition to be performed precisely while if the link $(s,i)$ does not exist for $i \in I$ and $s \in S_i$ but we will call $f_{si}$ as a {\it ghost flow} following \cite{dey2020convexifications}. Let us have a look at Figure~\ref{fig:SamplePoolingRelaxations} and try to write equation \eqref{sb:ghostflow} for $i=6$ and $s=1$:
\[x_{6,7}^1+x_{6,9}^1=f_{1,6}\]
Constraint \eqref{sb:specifications} is to meet the specification requirements at terminals. Constraint \eqref{sbbilinear}, which is the nonconvex bilinear constraint, calculates the fraction of flow at pool $i$ originated at source $s$ on each arc $(i,j)$. Finally, we have the nonnegativity of the variables in constraints \eqref{sb:x_nonnegativity} and \eqref{sb:q_nonnegativity}.

It is worth mentioning that the {\it Source-Based} formulation is equivalent to the {\it PQ}-formulation.  

\subsection{Polyhedral Relaxations} \label{sec:polyhedralRelaxSB}

The proposed exact formulation of the pooling problem is nonconvex. Now, we will present LP relaxations of this model. Our starting point will be the work of
\cite{dey2020convexifications} in which the bilinear constraint in the {\it Source-Based} formulation (\ref{sbbilinear}) is rewritten as a set of rank restrictions on a matrix consisting of decomposed flow variables $x_{ij}^s$ as follows:
\begin{equation}
    \text{rank} \left([x_{ij}^s]_{(s,j)\in S_i \times N_i^+}\right)\le 1 \qquad \forall i \in I. \label{sbrank1}
\end{equation}
As an example, consider pool $i=6$ in Figure~\ref{fig:SamplePoolingRelaxations}. It is easy to see that the following relation:
$$\begin{bmatrix}
        x_{6,7}^1 & x_{6,9}^1 \\
        x_{6,7}^2 & x_{6,9}^2 \\
        x_{6,7}^3 & x_{6,9}^3 \\
    \end{bmatrix} = \begin{bmatrix}
        q_6^1 \\
        q_6^2 \\
        q_6^3 \\
    \end{bmatrix} \times \begin{bmatrix}
        f_{6,7} & f_{6,9}\\
    \end{bmatrix}$$    
This example demonstrates the logic of the rank-one constraint as the matrix on the left-hand side is the 
bilinear constraint can be written as the product of a column vector and a row vector. 

The {rank} constraint \eqref{sbrank1} can be convexified in different ways. Below, we  present some LP-based relaxations in detail.

\subsubsection{Column-Wise Relaxation}
Let us consider constraints \eqref{sb:sourcecapacity}, \eqref{sb:arccapacity} (in which $f_{ij}$ is substituted by its equivalent values from \eqref{sb:ghostflow} and \eqref{sb:flowdecomposition} respectively), and the bilinear constraint \eqref{sbbilinear} (replaced with its equivalent {\it rank} constraint \eqref{sbrank1}) as a set. According to Theorem \ref{theorem:onlyRowOrColumn} and equation \eqref{eq:extended xt} we can define the \textit{column-wise} relaxation for the \textit{source-based} formulation for pool $i$ as below:
\begin{equation}\label{sb:colwiseRelax}
    \mathcal{F}_1^{\mathcal{S}(i)}:=\bigg \{ [x_{ij}^s]_{(s,j)\in S_i \times N_i^+} \in \conv(\tilde \cT(l_i,u_i, \cdot,\cdot,L_i,U_i) ) \bigg \}.
\end{equation}
This relaxation restricts the column-sum of the decomposed flow variables' matrices for all $i \in I$ and is equivalent to the McCormick relaxation of the {\it PQ}-formulation \citep{dey2020convexifications}.

As an illustration, let us consider $i=6$ in Figure~\ref{fig:SamplePoolingRelaxations} as a pool for which we will implement the column-wise extended relaxation. The associated sets for this pool are as follows:
\begin{equation}
    S_6=\{1,2,3\}, N^+_6=\{7,9\} \label{sets:sampleSB}
\end{equation}
Then, in the following matrix, we have a row for each element of $S_6$ and a column for each element in $N^+_6$. The bound of each column is the bound of an outgoing arc from the corresponding pool and the overall bound is the pool bound. This instance shows how we impose column-sum bounds on the matrix of decomposed flow variables for each pool.

\NiceMatrixOptions{code-for-first-row = \color{black},
                   code-for-first-col = \color{black},
                   code-for-last-row = \color{black},
                   code-for-last-col = \color{black}}
$$\left[x_{ij}^s\right]_{(s,j)}=\begin{pNiceArray}{cc}[first-row,last-row=4,first-col,last-col,nullify-dots]
    & u_{6,7}   &   u_{6,9}  & U_6\\
    & x_{6,7}^1 & x_{6,9}^1 \\
    & x_{6,7}^2 & x_{6,9}^2 \\
    & x_{6,7}^3 & x_{6,9}^3 \\
    L_6 & l_{6,7}    & l_{6,9}    
\end{pNiceArray}$$

\subsubsection{Row-Wise Relaxation}
Analogous to the \textit{column-wise} relaxation, the \textit{row-wise} relaxation can be defined based on Theorem \ref{theorem:onlyRowOrColumn} and equation \eqref{eq:extended xt'}. This relaxation, which restricts the row-sum of the decomposed flow variables' matrices for all $i \in I$, is defined as follows:

\begin{equation}\label{sb:rowwiseRelax}
    \mathcal{F}_2^{\mathcal{S}(i)}:=\bigg \{ [x_{ij}^s]_{(s,j)\in S_i \times N_i^+} \in \conv(\tilde \cT(\cdot,\cdot,l'_i,u'_i,L_i,U_i) ) \bigg \}.
\end{equation}

Considering pool 6 from Figure~\ref{fig:SamplePoolingRelaxations} and the associated sets defined in equations \eqref{sets:sampleSB}, the following matrix has a row for each element in $S_6$ and a column for each element of $N_6^+$. The bounds imposed on the summation of each row are the bounds of incoming arcs (including the {\it ghost flows}) to pool $i=6$ and the overall bound is the pool's capacity bounds. 

\NiceMatrixOptions{code-for-first-row = \color{black},
                   code-for-first-col = \color{black},
                   code-for-last-row = \color{black},
                   code-for-last-col = \color{black}}
$$\left[x_{ij}^s\right]_{(s,j)}=\begin{pNiceArray}{cc}[first-row,last-row=4,first-col,last-col,nullify-dots]
    &    &     & U_6\\
l_{1,6}    & x_{6,7}^1 & x_{6,9}^1 & u_{1,6}\\
l_{2,6}    & x_{6,7}^2 & x_{6,9}^2 & u_{2,6}\\
l_{3,6}    & x_{6,7}^3 & x_{6,9}^3 & u_{3,6}\\
    L_6 &     &     
\end{pNiceArray}$$
This matrix shows how we impose row bounds on the matrix of the decomposed flow variables.

\subsubsection{Intersection of Row-Wise and Column-Wise Relaxations}
We can use the intersection of the row-wise and column-wise relaxations as a new method to relax the nonlinear constraint of the pooling problem (equation \eqref{eq:intersectionRowCol}). As we saw in Section \ref{sec:intersectionRowCol}, this relaxation is at least as good as both of the previous ones but increases the scale of the problem. We use the extended formulations of the row-wise and column-wise relaxations to implement it and define it for all $i \in I$ as follows: 
\[\mathcal{F}_3^{\mathcal{S}(i)
}=\mathcal{F}_1^{\mathcal{S}(i)
} \cap \mathcal{F}_2^{\mathcal{S}(i)}\]

\subsection{Mixed-Integer Programming Approximations}
One of the other ways to deal with the bilinear constraint (\ref{sbbilinear}), is to utilize discretization methods. \cite{gupte2017relaxations} have classified the discretization methods proposed for the pooling problem into two different categories: i) forcing some variables to take certain prespecified values from their domain which applies to each bilinear program, and ii) discretizing the amount of flow at each pool which was proposed by \cite{dey2015analysis} for the first time and results in a ``network flow MILP restriction'' by exploiting the pooling problem's structure. Both of these strategies convert the pooling problem to an MILP and give an approximation of the problem. 
In this section, we will use the discretization methods described in \cite{dey2020convexifications}, which focus on the first strategy. We use them in the {\it Source-Based} formulation to obtain inner and outer approximations of the pooling problem.

In this section, we try to find an outer approximation (relaxation) by discretizing the proportion variables $q$ as follows:
\[q_j=\sum_{h=1}^H 2^{-h}z_{jh}+\gamma_j,\]
where $H \in \mathbb{Z}_{++}$ is the level of discretization, $z_{ih}$ are binary variables, and $\gamma_i$ is a continuous non-negative variable upper-bounded by $2^{-H}$. Now we define $x_{ij}^s$ as follows:
\[x_{ij}^s = \left(\sum_{j' \in N_i^+} x_{ij'}^s\right)\left(\sum_{h=1}^H 2^{-h}z_{jh}+\gamma_j \right)\qquad \forall i \in I, \forall s \in S_i, \forall j \in N_i^+.\]
Let $\alpha_{sjh}:=(\sum_{j' \in N_i^+} x_{ij'}^s)z_{jh}$ and $\beta_{ij}:=(\sum_{j' \in N_i^+}x_{ij'}^s)\gamma_j$, then, by using the McCormick envelopes, we obtain the following outer-approximation for all $i \in I,~ s \in S_i,~ \text{and}~ j \in N_i^+$:
\begin{equation*}
    \begin{split}
        \bar{\mathcal{D}}_{(|S_i|,|N_i^+|,H)}^{row} & ([l_{si}]_{s},[u_{si}]_{s}):= \\
        & \{x \in \mathbb{R}_+^{S_i \times N_i^+}| (\alpha,\beta, \gamma, z) \in \mathbb{R}^{S_i \times N_i^+ \times H} \times \mathbb{R}^{S_i \times N_i^+} \times \mathbb{R}^{N_i^+} \times \{0,1\}^{N_i^+ \times H}:
    \end{split}
\end{equation*}

\begin{align}
    & l_{si} z_{jh}\le \alpha _{sjh} \le u_{si} z_{jh} & \forall j \in N_i^+, \forall h \in [H],\label{sb:dis_relax_1}\\
    & u_{si} z_{jh}+\sum_{j' \in N_i^+} x_{ij'}^s - u_{si} \le l_{si} z_{jh} + \sum_{j' \in N_i^+} x_{ij'}^s - l_{si} & \forall j \in N_i^+, \forall h \in [H],\label{sb:dis_relax_2}\\
    & l_{si} \gamma_j \le \beta_{sj} \le u_{si} \gamma_j & \forall j \in N_i^+,\label{sb:dis_relax_3}\\
    \begin{split}
        & u_{si} \gamma_j + 2^{-H}(\sum_{j' \in N_i^+} x_{ij'}^s - u_{si}) \le \beta_{sj}\\
        & \beta_{sj} \le l_{si} \gamma_j + 2^{-H}(\sum_{j' \in N_i^+} x_{ij'}^s - l_{si})
    \end{split} &\forall j \in N_i^+,\label{sb:dis_relax_4}\\
    & l_{si} \le \sum_{j \in N_i^+} x_{ij}^s \le u_{si} \label{sb:dis_relax_5}\\
    & x_{ij}^s=\sum_{h=1}^H 2^{-h} \alpha_{sjh}+\beta_{sj} & \forall j \in N_i^+ \}.\label{sb:dis_relax_6}
\end{align}
\cite{dey2020convexifications} showed that $\bar{\mathcal{D}}_{(|S_i|,|N_i^+|,H)}^{row}$ is a relaxation of the source-based rank formulation. 
 
We can analogously define the outer approximation denoted as $\bar{\mathcal{D}}_{(|S_i|,|N_I^+|,H)}^{col}([l_{ij}]_j,[u_{ij}]_j), \forall i \in I$ by restricting the sum of each column in the decomposed flow variable matrices and as well. 


Analogous to the {\it Source-Based} formulation, we have the {\it Terminal-Based} formulation consisting of the proportion variables corresponding to the terminals and the flow variables along with the arcs between sources and pools. This model was first introduced in \cite{alfaki2013multi} as the {\it TP}-formulation and was later utilized in \cite{dey2020convexifications}. 


\section{Solution Approach}\label{sec:solution}
In this section, we develop a new LP relaxation that considers imposing bounds on the row-sum and the column-sum of the decomposed flow variable matrices simultaneously (Section \ref{sec:newLPRelax}). We also provided the technical details of obtaining new valid inequalities by the RLT in Section \ref{sec:validIneq}, which we will use in the computations. 
In addition, we discuss how to utilize the OBBT technique to improve the bounds of the arcs and nodes of the generalized pooling problem instances of the literature (Section \ref{sec:OBBT}). Moreover, we propose a simple and computationally cheap bound tightening method to improve the bounds of the mining problem as a special case of the generalized pooling problem with the ``time-indexed'' feature (Section \ref{sec:miningTopoolingInstruction}). {It is important to note that this section represents a novel application of the concepts and techniques introduced in Section \ref{sec:socp} to the pooling problem, showcasing the versatility and effectiveness of our approach.}

\subsection{New Relaxations and Valid Inequalities}
We have discussed the well-known relaxations of the literature in Section \ref{sec:formulations}. In this section, we will present the relaxations and valid inequalities we have developed to improve the quality of the dual bounds.
\subsubsection{New Linear Programming Relaxations}\label{sec:newLPRelax}
We have reviewed the row-wise and column-wise extended relaxations for the {\it Source-Based} and {\it Terminal-Based} multi-commodity flow formulations presented by \cite{dey2020convexifications} in Section \ref{sec:polyhedralRelaxSB}. Also, we discussed that a stronger relaxation can be obtained by intersecting these two aforementioned relaxations, which may increase the size of the problem. In Section \ref{sec:strongerPolyRelax}, we showed that to have an even stronger relaxation, we can consider imposing bounds on the row-sum and column-sum of a matrix consisting of the decomposed flow variables of each pool simultaneously. We call this relaxation {\it Row-Column} and define it as the following.

\begin{equation}\label{sb:rowcolumnRelax}
    \mathcal{F}_4^{\mathcal{S}(i)}:=\bigg \{ [x_{ij}^s]_{(s,j)\in S_i \times N_i^+} \in \cT^2(l_i,u_i,l'_i,u'_i,L_i,U_i) \bigg \}.
\end{equation}

This relaxation convexifies the set $\tilde \cT^2(l_i,u_i,l'_i,u'_i,L_i,U_i)$ defined in the equation \eqref{eq:extended xr + rank} which can be considered as the intersection of the bilinear constraint (i.e., its equivalent {\it rank} constraint) and the capacity constraints that we have in the {\it Source-Based} formulation for each pool $i \in I$.

Proposition \ref{prop: T T1} shows the \textit{row-column} relaxation is at least as good as the intersection of the row-wise and the column-wise relaxations. Therefore, we have the following in which the left-hand side relationship can be strict:
\[\mathcal{F}_4^{\mathcal{S}(i)} \subseteq \mathcal{F}_3^{\mathcal{S}(i)}:= \mathcal{F}_1^{\mathcal{S}(i)} \cap \mathcal{F}_2^{\mathcal{S}(i)}\]

{
\subsection{Valid Inequalities}
In Section \ref{sec:validIneq}, we applied the RLT to derive new valid inequalities to strengthen the relaxations. Now, we will exemplify how these inequalities  are adaptable to the context and the notations of the pooling problem. For brevity, we will use equations \eqref{eq:validac_r} and \eqref{eq:validac_x} below as examples although all the inequalities derived in Section \ref{sec:validIneq} are applicable in principle.
\begin{itemize}
    \item The inequalities \eqref{eq:validac_r} in $r$ variables can be written as follows:
    \begin{equation*}
    \begin{split}
          u_{si} \bigg[ (u_{ij}/L_i+l_{ij}/U_i)  \sum_{s' \in S_i} r_{ij}^{s'} - (u_{ij}l_{ij})/(U_iL_i)  \bigg] \ge (u_{si} l_{ij}/U_i) \sum_{s' \in S_i} r_{ij}^{s'} \\
          - (l_{ij}^2 /U_i) \sum_{j' \in N_i^+} r_{ij'}^s + l_{ij} r_{ij}^s\\
    \end{split}
    \end{equation*}
    \item The corresponding inequalities \eqref{eq:validac_x} in $x$ variables are as follows:
    \begin{equation*}
    \begin{split}
    u_{si} \bigg[ (u_{ij}/L_i+l_{ij}/U_i)  \sum_{s' \in S_i} x_{ij}^{s'} - (u_{ij}l_{ij})/(U_iL_i)  \sum_{s' \in S_i}\sum_{j' \in N_i^+} x_{ij'}^{s'}  \bigg] & \ge \\ 
    (u_{si} l_{ij}/U_i) \sum_{s' \in S_i} x_{ij}^{s'} - (l_{ij}^2 /U_i) \sum_{j' \in N_i^+} x_{ij'}^s + l_{ij} x_{ij}^s
    \end{split}
    \end{equation*}
\end{itemize}
Our experiments involve simultaneously adding  valid inequalities  both in  $x$ and $r$ variables. Specifically, we denote the valid inequalities resulting from the multiplication of \eqref{eq:extt1} and \eqref{eq:extt2} as $\mathcal{V}_{ab}$, while the ones obtained by multiplying \eqref{eq:extt1} and \eqref{eq:extt3} are labeled as $\mathcal{V}_{ac}$. We have restricted ourselves to these families of inequalities since the other inequalities either have a negligible effect on the overall results or cause numerical issues. Detailed results of our experiments with the addition of valid inequalities can be found in Section \ref{sec:comp-valid}.
}

\subsection{Bound Tightening}

\subsubsection{Optimization-Based Bound Tightening}\label{sec:OBBT}
In global optimization, one of the valuable tools to reduce the variables' domain is to execute OBBT \citep{coffrin2015strengthening,puranik2017bounds,bynum2018tightening}. Let \underbar{$z$} and $\bar{z}$ represent the lower and upper bound of our multi-commodity flow problems, which can be obtained from a relaxation and any primal solution, respectively. Then to find the bounds of arcs and different nodes, we optimize the following objective functions over a linear programming relaxation of the original problem. We need to take into account that the original objective function should be between \underbar{$z$} and $\bar{z}$. Therefore, we add this as a constraint to the new problem. 
\begin{itemize}
    \item To find the lower bound (upper bound) of each arc $(i,j) \in A$, we minimize (maximize) the corresponding flow variable $(f_{ij})$.
    \item To generate a lower bound (upper bound) for each source node $s$, we minimize (maximize) the summation of outgoing flows from that node $(\sum_{i \in N^+_s} f_{si})$.
    \item For each pool $i$, to find a lower bound (upper bound), we minimize (maximize) the summation of incoming flows $(\sum_{j \in N^-_i} f_{ji})$ or the summation of outgoing flows from that node $(\sum_{j \in N^+_i} f_{ij})$.
    \item Finally, to improve the lower bound (upper bound) of each terminal node $t$, we minimize (maximize) the summation of all the incoming flows to that terminal $(\sum_{i \in N^-_t} f_{it})$.
\end{itemize}

\subsubsection{Bound Tightening for the Time-indexed Pooling Problem} \label{sec:miningTopoolingInstruction}
A special case of the generalized pooling problem which arises in the mining industry is investigated in \cite{boland2015special}. In this case, the raw material (supply) with certain specifications comes into stockpile $p=\{1, \dots, P\}$ at time $\tau \in \mathcal{T}_p^s$. On the other hand, demand for the final product with the desired specifications is placed at time $\tau \in \mathcal{T}^t$. Any violations of the output specifications from the customer's desired ones will cause a ``contractually agreed'' penalty, and the objective function is to minimize this penalty. The prescription of converting this problem to a general pooling problem by \cite{boland2015special} is as follows.

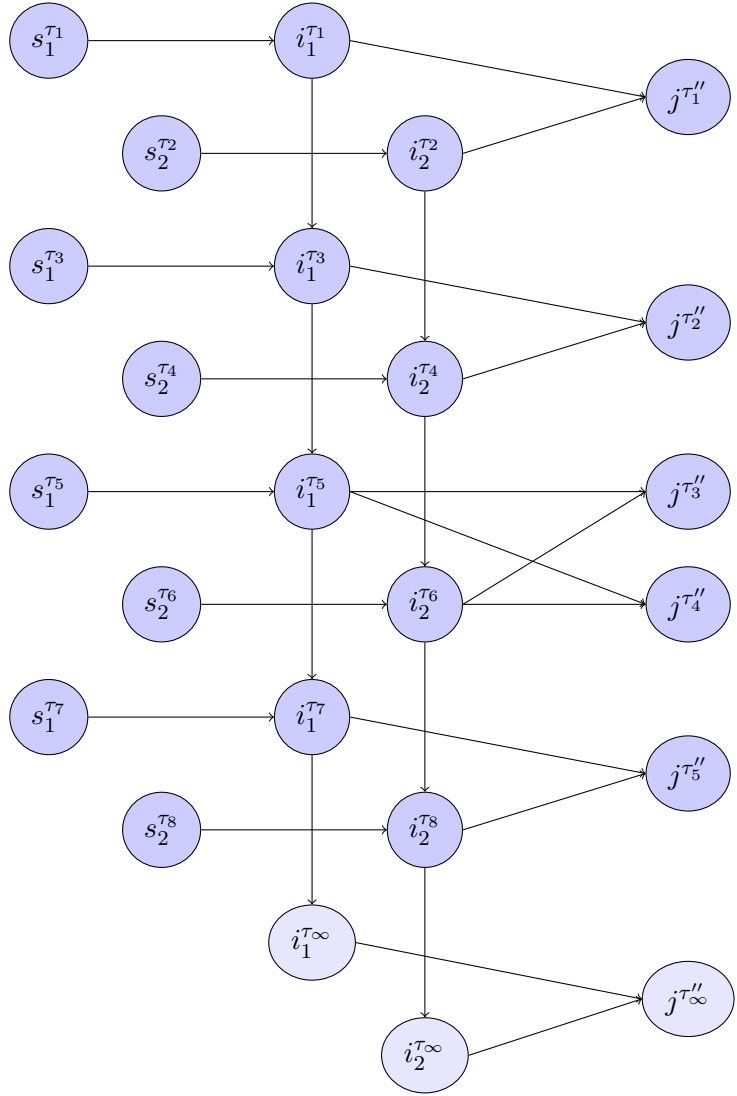
\begin{figure}
    \begin{center}
        \begin{tikzpicture}[node distance=1.5cm, circle/.style={draw, circle, inner sep=0pt, minimum size=12mm}]
        \tikzstyle{input} = [draw, ellipse, minimum height=1cm, minimum width=1cm, fill=blue!20]
        \tikzstyle{output} = [draw, ellipse, minimum height=1cm, minimum width=1cm, fill=blue!20]
        \tikzstyle{pool} = [draw, ellipse, minimum height=1cm, minimum width=1cm, fill=blue!20]
        \node [input] (input1) at (0,0) {$s_1^{\tau_1}$};
        \node [input] (input2) at (1.5,-1.5) {$s_2^{\tau_2}$};
        \node [input] (input3) at (0,-3) {$s_1^{\tau_3}$};
        \node [input] (input4) at (1.5,-4.5) {$s_2^{\tau_4}$};
        \node [input] (input5) at (0,-6) {$s_1^{\tau_5}$};
        \node [input] (input6) at (1.5,-7.5) {$s_2^{\tau_6}$};
        \node [input] (input7) at (0,-9) {$s_1^{\tau_7}$};
        \node [input] (input8) at (1.5,-10.5) {$s_2^{\tau_8}$};
        
        \node [pool] (pool1) at (3.5,0) {$i_1^{\tau_1}$};
        \node [pool] (pool2) at (5,-1.5) {$i_2^{\tau_2}$};
        \node [pool] (pool3) at (3.5,-3) {$i_1^{\tau_3}$};
        \node [pool] (pool4) at (5,-4.5) {$i_2^{\tau_4}$};
        \node [pool] (pool5) at (3.5,-6) {$i_1^{\tau_5}$};
        \node [pool] (pool6) at (5,-7.5) {$i_2^{\tau_6}$};
        \node [pool] (pool7) at (3.5,-9) {$i_1^{\tau_7}$};
        \node [pool] (pool8) at (5,-10.5) {$i_2^{\tau_8}$};
        \node [pool,fill=blue!10] (pool9) at (3.5,-12) {$i_1^{\tau_\infty}$};
        \node [pool,fill=blue!10] (pool10) at (5,-13.5) {$i_2^{\tau_\infty}$};
    
        \node [output] (output1) at (8.5,-0.75) {$j^{\tau''_1}$};
        \node [output] (output2) at (8.5,-3.75) {$j^{\tau''_2}$};
        \node [output] (output3) at (8.5,-6) {$j^{\tau''_3}$};
        \node [output] (output4) at (8.5,-7.5) {$j^{\tau''_4}$};
        \node [output] (output5) at (8.5,-9.75) {$j^{\tau''_5}$};
        \node [output,fill=blue!10] (output6) at (8.5,-12.75) {$j^{\tau''_\infty}$};
        
        \draw[->] (input1.east) -- (pool1.west);
        \draw[->] (input2.east) -- (pool2.west);
        \draw[->] (input3.east) -- (pool3.west);
        \draw[->] (input4.east) -- (pool4.west);
        \draw[->] (input5.east) -- (pool5.west);
        \draw[->] (input6.east) -- (pool6.west);
        \draw[->] (input7.east) -- (pool7.west);
        \draw[->] (input8.east) -- (pool8.west);

        \draw[->] (pool1.south) -- (pool3.north);
        \draw[->] (pool2.south) -- (pool4.north);
        \draw[->] (pool3.south) -- (pool5.north);
        \draw[->] (pool4.south) -- (pool6.north);
        \draw[->] (pool5.south) -- (pool7.north);
        \draw[->] (pool6.south) -- (pool8.north);
        \draw[->] (pool7.south) -- (pool9.north);
        \draw[->] (pool8.south) -- (pool10.north);
        
        \draw[->] (pool1.east) -- (output1.west);
        \draw[->] (pool2.east) -- (output1.west);
        \draw[->] (pool3.east) -- (output2.west);
        \draw[->] (pool4.east) -- (output2.west);
        \draw[->] (pool5.east) -- (output3.west);
        \draw[->] (pool5.east) -- (output4.west);
        \draw[->] (pool6.east) -- (output3.west);
        \draw[->] (pool6.east) -- (output4.west);
        \draw[->] (pool7.east) -- (output5.west);
        \draw[->] (pool8.east) -- (output5.west);
        \draw[->] (pool9.east) -- (output6.west);
        \draw[->] (pool10.east) -- (output6.west);
        
        \end{tikzpicture}
    \end{center}
    \caption{A Mining Problem Instance}
    \label{fig:miningNoBounds}
\end{figure}

\begin{itemize}
    \item {\bf Input Nodes}: Create the input node $s_p^\tau$ for each supply coming into stockpile $p$ at time $\tau \in \mathcal{T}_p^s$.
    \item {\bf Pool Nodes}: Create the pool node $i_p^\tau$ for each supply coming into stockpile $p$ at time $\tau \in \mathcal{T}_p^s$.
    \item {\bf Output Nodes}: Create the output node $j^\tau$ for each demand at time $\tau \in \mathcal{T}^t$.
    \item {\bf Input-to-Pool Arcs}: Create an arc from input node $s_p^\tau$ to pool node $i_p^\tau$ for each supply coming into stockpile $p$ at time $\tau \in \mathcal{T}_p^s$.
    \item {\bf Pool-to-Pool Arcs}: Create an arc from $i_p^\tau$ to $i_p^{\tau'}$ where $\tau' \in \mathcal{T}_p^s$ is the time of the ``immediate successor'' supply of the one at time $\tau \in \mathcal{T}_p^s$ coming to the stockpile $p$.
    \item {\bf Pool-to-Output Arcs}: Create an arc from $i_p^\tau$ to $j^{\tau''}$ where $\tau'' \in \mathcal{T}^t$ is the time of the ``immediate successor'' demand of the supply $s_p^\tau$ coming to stockpile $p$ at time $\tau \in \mathcal{T}_p^s$.
\end{itemize}

We also add one extra pool with time $\tau=\infty$ for the supply surplus of each stockpile whose summation is all being directed to an extra output node that we add. The amount of incoming flow to the last output equals the summation of all the supplies minus the summation of all the demands. Figure~\ref{fig:miningNoBounds} shows a mining problem instance.

Since the mining problem is mostly large-scale and has some unique features, such as being time-indexed, which increases the size of the problem even more, it may not be a good idea to perform Optimization-Based Bound Tightening on it. Therefore, we propose some simple and cheap methods to improve the bounds of this problem.
Let us consider Figure~\ref{fig:miningNoBounds} as a part of a mining problem instance which is formulated as a general pooling problem in which the nodes are placed vertically to reflect the time of supply/demand (the supply/demand node's names represent the ordering of their time). In addition, there are two stockpiles in this example, and we can see the input and pool nodes are aligned in two lines to show which nodes are from the same stockpile (even supply and pool nodes are from stockpile $1$ and odds are from stockpile $2$). Algorithm \ref{alg:miningBoundTightening} shows how we can improve the bounds of different nodes and arcs of this instance.\\

\begin{algorithm}[H]
\caption{Bound Tightening for the Mining Problem}
\label{alg:miningBoundTightening}
\begin{algorithmic}[1]
\STATE Bounds of each input node and its outgoing arc equals its amount of supply;
$$L_s=U_s=l_{si}=u_{si}=q_s \qquad \forall s \in S, i \in N_s^+.$$
\STATE Bounds of each output node equals its amount of demand;
$$L_t=U_t=d_t \qquad \forall t \in T.$$
\STATE The lower bound of each pool-to-terminal arc is improved as:
$$l_{it}=\max\{L_t-\sum_{p \in P_i} U_p,0\} \qquad \forall (i,t) \in A, i \in I, t \in T.$$
The upper bound of this arc is improved as:
$$u_{it}=\min \{U_i,U_t\} \qquad \forall (i,t) \in A, i \in I, t \in T.$$
\STATE The lower bound of each pool-to-pool arc is calculated by:
$$l_{ij}=\max \{L_i - \sum_{(i,t) \in A, t \in T} U_t, 0\} \qquad \forall (i,j)\in A, i,j \in I,$$
To improve the upper bound of this arc, we calculate the following:
$$u_{ij}=U_i-\sum_{(i,t) \in A, t \in T} l_{it} \qquad \forall (i,j)\in A, i,j \in I,$$
In addition to the amount of supply surplus after meeting each demand. The upper bound of each pool-to-pool arc equals the minimum between the supply surplus and $u_{ij}$.
\STATE The bounds of each pool equals the summation of the bounds of all its incoming arcs;
$$L_i=\sum_{j \in N^-_i} l_{ji}~\text{ and }~U_i=\sum_{j \in N^-_i} u_{ji} \qquad \forall i \in I.$$
\end{algorithmic}
\end{algorithm}

\section{Computations}\label{sec:computations}
In this section, we present the results of our experiments on two different sets of generalized pooling problem instances. First, we consider 13 well-known standard pooling problem instances from the literature \citep{haverly1978studies,adhya1999lagrangian,foulds1992bilinear,ben1994global}. We have generalized them by adding the arcs $(i, j)$ and $(j, i)$ for each pair of pools $i, j \subseteq I$, where $i \ne j$ \citep{alfaki2013multi}. Second, we use the data of the real-world mining problem instances based on the work of \cite{boland2015special}. In what follows, we report the results of the exact methods based on the original {\it Source-Based} and {\it Terminal-Based} rank formulations as well as the outcomes of the experiments with  different types of relaxations, restrictions, and valid inequalities which we discussed in Sections \ref{sec:formulations} and \ref{sec:solution}. We perform all the experiments with and without the bound tightening and report the results separately. Table \ref{tab:computationalMethods} shows the methods and notations we use.\\

\begin{table}[H]
  \centering
  \caption{Computational Methods}
  \scalebox{0.9}{
    \begin{tabular}{clcc}
    \toprule
    \multicolumn{2}{c}{\multirow{2}[4]{*}{Method}} & \multicolumn{2}{c}{Notation} \\
\cmidrule{3-4}    \multicolumn{2}{c}{} & \textit{Source-Based} & \textit{Terminal-Based} \\
    \midrule
    \multirow{4}[2]{*}{LP Relaxations} & \textit{Column-wise} &  $\mathcal{F}_1^\mathcal{S}$   & $\mathcal{F}_2^\mathcal{T}$ \\
          & \textit{Row-wise} & $\mathcal{F}_2^\mathcal{S}$    & $\mathcal{F}_1^\mathcal{T}$ \\
          & \textit{Row-wise} $\cap$ \textit{Column-wise} & $\mathcal{F}_3^\mathcal{S}$    & $\mathcal{F}_3^\mathcal{T}$ \\
          & \textit{Row-column*} & $\mathcal{F}_4^\mathcal{S}$    & $\mathcal{F}_4^\mathcal{T}$ \\
\cmidrule{1-4}    \multirow{2}[2]{*}{MIP Relaxations} & Discretizing $q$ considering $X$'s column-sum & $\mathcal{M}_1^\mathcal{S}(H)$   & $\mathcal{M}_2^\mathcal{T}(H)$ \\
          & Discretizing $q$ considering $X$'s row-sum & $\mathcal{M}_2^\mathcal{S}(H)$   & $\mathcal{M}_1^\mathcal{T}(H)$ \\
\cmidrule{1-4}    \multirow{2}[2]{*}{Valid Inequalities} & Obtained by multiplication of \eqref{eq:extt1} and \eqref{eq:extt2}* & $\mathcal{V}_{ab}^\mathcal{S}$  & $\mathcal{V}_{ab}^\mathcal{T}$ \\
          & Obtained by multiplication of \eqref{eq:extt1} and \eqref{eq:extt3}* & $\mathcal{V}_{ac}^\mathcal{S}$  & $\mathcal{V}_{ac}^\mathcal{T}$ \\
    \midrule
    \multicolumn{4}{c}{* Developed in this paper}\\
    \end{tabular}}%
  \label{tab:computationalMethods}%
\end{table}%

All the experiments are implemented in Python 3.7, and optimization problems are solved by Gurobi 9.1.1 on an Intel(R) 3.7 GHz processor and 64 GB RAM workstation. {The time limit for each experiment is set to one hour}. Also, we have utilized the Python \texttt{JobLib} package to perform   OBBT in parallel for each pair of $(i,j) \in A$ and node $i \in N$.

\subsection{Literature Instances}
In this section, we focus on the instances from the literature and perform different experiments. Table \ref{tab:litInsChar} shows the total number of input, pool, and output nodes as well as the total number of arcs and specifications in each of the generalized instances from the literature.

\subsubsection{Exact Formulations}
First, we solve these instances with the original {\it Source-Based} and {\it Terminal-Based} formulations. Table \ref{tab:sumLitExact} shows these results (additional details in Appendix, Tables \ref{tab:litSBexact} and \ref{tab:litTBexact}). 

\begin{table}[H]
  \centering
  \caption{Literature Instances: Exact Formulations}
  \scalebox{0.9}{
    \begin{tabular}{ccccccc}
    \toprule
    \multirow{2}[4]{*}{Formulation} & \multicolumn{2}{c}{Gurobi without OBBT} &       & \multicolumn{3}{c}{Gurobi with OBBT} \\
\cmidrule{2-3}\cmidrule{5-7}          & Time  & \%$O$-Gap &       & Prep. Time* & Time  & \%$O$-Gap \\
    \midrule
    \textit{Source-Based} & 1117.10 & 0.66\% &       & 18.62 & 185.45 & 0.00\% \\
    \textit{Terminal-Based} & \textbf{277.43} & 0.67\% &       & 18.62 & \textbf{2.65}  & 0.00\% \\
    \midrule
    \multicolumn{7}{c}{* Preprocessing Time} \\
    \end{tabular}}%
  \label{tab:sumLitExact}%
\end{table}%

This table presents the running time and the optimality gap ($O$-Gap) when using the Gurobi solver to obtain the `Exact' solution for the literature instances. It compares the results with and without the OBBT technique. The preprocessing time refers to the overall time taken to compute the bounds for the original objective value plus the OBBT processing time.

In OBBT, to obtain a lower bound for the objective value, we have solved the original {\it Terminal-Based} formulation in which we have relaxed the bilinear constraint and refer to it as the Multi-Commodity Flow (MCF) formulation in the rest of the paper. In addition, to get an upper bound, we have solved the restriction $\mathcal{G}_2^\mathcal{T}(H=3)$ developed by \cite{dey2020convexifications} for all the instances.

According to the table, the OBBT technique helps to improve the running time and the optimality gap for most of the instances for both formulations. In terms of the running time, the {\it Terminal-Based} formulation performs better than the {\it Source-based} formulation even without the bound tightening. Additionally, in the {\it Terminal-Based} formulation, the Gurobi is able to close the gap in a much shorter time than the previous formulation.

Generally, we can say that performing the OBBT technique on the generalized version of literature instances is advantageous on average, and the time and optimality gap improvements are more significant for the {\it Terminal-Based} formulation.

\subsubsection{Linear Programming Relaxations}
In this section, we investigate the performance of different relaxations we have explained for the pooling problem instances of the literature. Also, we report the results of these methods with their original bounds and with the improved bounds to evaluate the effect of OBBT on these outer approximations.

\begin{table}[H]
  \centering
  \caption{Literature Instances \textbf{without OBBT} (LP Relaxations)}
  \scalebox{0.9}{
    \begin{tabular}{cccccccccccc}
    \toprule
    \multirow{2}[4]{*}{Formulation} & \multicolumn{2}{c}{$\mathcal{F}_1$} &       & \multicolumn{2}{c}{$\mathcal{F}_2$} &       & \multicolumn{2}{c}{$\mathcal{F}_3$} &       & \multicolumn{2}{c}{$\mathcal{F}_4$} \\
\cmidrule{2-3}\cmidrule{5-6}\cmidrule{8-9}\cmidrule{11-12}          & Time & \% $D$-Gap &       & Time & \% $D$-Gap &       & Time & \% $D$-Gap &       & Time & \% $D$-Gap \\
    \midrule
    \textit{Source-Based} & 0.04  & 15.65\% &       & 0.03  & 15.72\% &       & 0.15  & 15.65\% &       & 0.15  & 15.65\% \\
    \textit{Terminal-Based} & \textbf{0.02}  & 15.65\% &       & \textbf{0.02}  & \textbf{14.60\%} &       & \textbf{0.03}  & \textbf{14.53\%} &       & \textbf{0.03}  & \textbf{14.53\%} \\
    \bottomrule
    \end{tabular}}%
  \label{tab:litLPnoOBBT}%
\end{table}%

Table \ref{tab:litLPnoOBBT} shows the results of the LP relaxations without the bound improvements (additional details in Appendix, Tables \ref{tab:litSBLPnoOBBT} and \ref{tab:litTBLPnoOBBT}). This table indicates the running time and the duality gap ($D$-Gap). To calculate this gap, we consider the objective values of the `Exact' obtained by `Gurobi with OBBT' with $0.00\%$ optimality gap as the upper bound ($UB$) and the bounds obtained by the relaxations as the lower bound ($LB$) and use the following equation:
\begin{equation} \label{eq:Gap}
    \text{Gap}=\frac{UB-LB}{|UB|} \times 100
\end{equation}
In general, the average running time of the LP relaxations is much smaller than `Exact'. Without performing OBBT, all the LP relaxations of both formulations yield almost the same duality gap while the {\it Terminal-Based} formulation is able to give slightly better duality gap percentages. As we can see, $\mathcal{F}_3$ is better than $\mathcal{F}_1$ and $\mathcal{F}_2$, and interestingly it is equal to $\mathcal{F}_4$ here.

\begin{table}[H]
  \centering
  \caption{Literature Instances \textbf{with OBBT} (LP Relaxations)}
  \scalebox{0.9}{
    \begin{tabular}{cccccccccccc}
    \toprule
    \multirow{2}[4]{*}{Formulation} & \multicolumn{2}{c}{$\mathcal{F}_1$} &       & \multicolumn{2}{c}{$\mathcal{F}_2$} &       & \multicolumn{2}{c}{$\mathcal{F}_3$} &       & \multicolumn{2}{c}{$\mathcal{F}_4$} \\
\cmidrule{2-3}\cmidrule{5-6}\cmidrule{8-9}\cmidrule{11-12}          & Time & \% $D$-Gap &       & Time & \% $D$-Gap &       & Time & \% $D$-Gap &       & Time & \% $D$-Gap \\
    \midrule
    \textit{Source-Based} & 0.05  & 7.99\% &       & \textbf{0.01}  & \textbf{4.96\%} &       & \textbf{0.02}  & 4.36\% &       & 0.03  & 4.35\% \\
    \textit{Terminal-Based} & \textbf{0.01}  & \textbf{4.96\%} &       & 0.02  & 7.99\% &       & 0.03  & 4.36\% &       & 0.03  & 4.35\% \\
    \bottomrule
    \end{tabular}}%
  \label{tab:litLPOBBT}%
\end{table}%

In Table \ref{tab:litLPOBBT}, we can see the results of the LP relaxations of the {\it Source-Based} and {\it Terminal-Based} formulations with OBBT (additional details in Appendix, Tables \ref{tab:litSBLPOBBT} and \ref{tab:litTBLPOBBT}). According to the results, the {\it column-wise} relaxations of the {\it Source-Based} and {\it Terminal-Based} formulations ($\mathcal{F}_1^\mathcal{S}$ and $\mathcal{F}_2^\mathcal{T}$, respectively) have the same performance while solving the literature instances of the pooling problem and give the duality gap percentage of $7.99$ on average. Identically, the {\it row-wise} relaxations of these two formulations ($\mathcal{F}_2^\mathcal{S}$ and $\mathcal{F}_1^\mathcal{T}$) give the same solution qualities with the average duality gap percentage of $4.96$. In addition, the intersection of the {\it row-wise} and {\it column-wise} ($\mathcal{F}_3$) and the {\it row-column} relaxation ($\mathcal{F}_4$) perform better than the previous LP relaxations which just consider imposing bounds on the row-sum or the column-sum. Moreover, $\mathcal{F}_3$ and $\mathcal{F}_4$ are interestingly equal for this data set.

We can realize that performing OBBT on the literature instances before utilizing the LP relaxations to solve them improves the duality gap significantly and this improvement is more significant than that of the `Exact' solutions. 

\subsubsection{Discretization Relaxations}
In this section, we evaluate the performance of the MIP relaxations in solving the generalized version of the literature instances. We compare the results of these methods before and after applying the OBBT  technique while discretizing the variable $q$. \cite{dey2020convexifications} have investigated the impact of the different discretization levels $H=1,\dots,5$ and shown that a good choice for the pooling problem that balances accuracy and computational effort is $H=3$. Therefore, we have considered the same level to run the experiments using MIP relaxations and restrictions. 

\begin{table}[H]
  \centering
  \caption{Literature Instances: MIP Relaxations ($H=3$)}
  \scalebox{0.9}{
    \begin{tabular}{cccccccccccc}
    \toprule
    \multirow{2}[4]{*}{OBBT} & \multicolumn{2}{c}{$\mathcal{M}_1^\mathcal{S}(H)$} &       & \multicolumn{2}{c}{$\mathcal{M}_2^\mathcal{S}(H)$} &       & \multicolumn{2}{c}{$\mathcal{M}_1^\mathcal{T}(H)$} &       & \multicolumn{2}{c}{$\mathcal{M}_2^\mathcal{T}(H)$} \\
\cmidrule{2-3}\cmidrule{5-6}\cmidrule{8-9}\cmidrule{11-12}          & Time & \% $D$-Gap &       & Time & \% $D$-Gap &       & Time & \% $D$-Gap &       & Time & \% $D$-Gap \\
    \midrule
    \textit{No} & 17.38  & 8.35\% &       & 136.38  & 0.84\% &       & \textbf{0.12}  & 0.61\% &       & 0.14  & 0.14\% \\
    \textit{Yes} & \textbf{0.29}  & \textbf{0.24\%} &       & \textbf{4.06}  & \textbf{0.11\%} &       & 0.14  & \textbf{0.24\%} &       & 0.14  & \textbf{0.11\%} \\
    \bottomrule
    \end{tabular}}%
  \label{tab:sumlitMIPrelax}%
\end{table}%

Table \ref{tab:sumlitMIPrelax} shows the results of the different discretization relaxations for the pooling problem instances of the literature with and without OBBT (additional details in Appendix, Tables \ref{tab:litMIPrelaxnoOBBT} and \ref{tab:litMIPrelaxOBBT}). In case of having no OBBT, $\mathcal{M}_1^\mathcal{S}(H)$ cannot perform as well as the others in terms of the solution quality and gives an average gap percentage of $8.35$ for all the instances. In terms of the running time, the MIP relaxations of the {\it Source-Based} formulation are not as strong as those of the {\it Terminal-Based} formulation and take more time to solve the problems. This difference is more significant when comparing  $\mathcal{M}_2^\mathcal{S}(H)$ to the others.

However, we observe that OBBT helps the MIP relaxations of the {\it Source-Based} formulation to have remarkable improvements in terms of the running time and the duality gap on average as well (additional details in Appendix, Table \ref{tab:litMIPrelaxOBBT}).

Regarding the use of discretization relaxations for the literature instances, OBBT does not make remarkable improvements for the relaxations of the {\it Terminal-Based} formulation since it has a good performance already. Utilizing this bound improvement method is more beneficial for the relaxations of the {\it Source-Based} formulation. It helps the model to obtain better bounds in a shorter time. 

Generally, MIP relaxations are stronger than the LP methods on average, but they are computationally more expensive and need more time to reach high-quality dual bounds.

\subsection{Mining Instances}
In this section, we report the results of applying different methods we have described previously to solve real-world cases of the mining problem. We have converted these problems to the generalized pooling problem by the instructions in Section \ref{sec:miningTopoolingInstruction}. This set consists of yearly, half-yearly, and quarterly planning time horizons. Table \ref{tab:MiningInsChar} shows the characteristics of the different instances in this dataset. This table shows the number of sources, pools, and terminals plus the overall number of arcs. The number of source-to-pool arcs is indicated by $|ASI|$, and we have used similar notations for pool-to-pool and pool-to-terminal arcs. The number of specifications is the same for all the instances. Supply of the raw materials is coming to two stockpiles indicated by $SP 1$ and $SP 2$ at different time points. 

The supplies of the raw materials must be blended in the pools and mixed again in the output points to meet the demand amount with certain specification requirements. There are four specifications; ash, moisture, sulfur, and volatile, which should not violate the maximum preferable amount specified by the customers. Otherwise, the supplier will be penalized by a contractually agreed amount, and the objective is to minimize this penalty.

\subsubsection{Exact Formulations}
Table~\ref{tab:sumMiningExactnoBound} shows the `Exact' objective value of solving the mining instances (additional details in Appendix, Tables \ref{tab:MiningExactnoBound} and \ref{tab:MiningExactBound}).

\begin{table}[H]
  \centering
  \caption{Mining Instances: Exact Formulations}
  \scalebox{0.9}{
    \begin{tabular}{ccccccc}
    \toprule
    \multirow{2}[4]{*}{Formulation} & \multicolumn{2}{c}{Gurobi} &       & \multicolumn{3}{c}{Gurobi with Bounds} \\
\cmidrule{2-3}\cmidrule{5-7}          & Time  & \%$O$-Gap &       &  Time  & \%$O$-Gap \\
    \midrule
    \textit{Source-Based} & 1384.84 & 3.13\% &       &  1767.06 & 2.79\% \\
    \textit{Terminal-Based} & 1690.87 & \textbf{1.25\%} &       &  1520.46  & \textbf{0.88\%} \\
    \bottomrule
    \end{tabular}}%
  \label{tab:sumMiningExactnoBound}%
\end{table}%

According Table~\ref{tab:sumMiningExactnoBound}, while using the {\it Source-Based} formulation, the optimality gap and the running time of the Gurobi are $3.13\%$ and $1457.08$ on average, respectively. On the other hand, the {\it Terminal-Based} formulation finds the solutions with the optimality gap of $1.25\%$ in the running time of $1690.87$ on average.
Additionally, the average optimality gap of the {\it Source-Based} formulation, while having updated bounds, has decreased. Moreover, updating the bounds of the problem has a positive impact on the running time and the optimality gap of Gurobi while experimenting with the {\it Terminal-Based} formulation.

\subsubsection{Linear Programming Relaxations}
In this section, we utilize the LP relaxations to deal with the generalized pooling problem counterpart of the mining problem instances.

\cite{boland2015special} have used the McCormick envelopes to obtain dual bounds of the mining instances. They have also modeled and solved the mining problem (in addition to its generalized pooling problem counterpart) and reported the best-known primal bounds for these instances. The authors have calculated the duality gap of their relaxation based on the primal bounds they have obtained. Since their primal bounds are cheaper and to have comparable results, we have reported their gap in the tables consisting of the results of our relaxations and calculated the duality gap of the results based on their primal bounds.

\begin{table}[H]
  \centering
  \caption{Mining Instances \textbf{without Bounds}: LP Relaxations}
  \scalebox{0.8}{
    \begin{tabular}{cccccccccccccc}
    \toprule
    \multirow{2}[4]{*}{Formulation} & $D$-Gap   &       & \multicolumn{2}{c}{$\mathcal{F}_1$} &       & \multicolumn{2}{c}{$\mathcal{F}_2$} &       & \multicolumn{2}{c}{$\mathcal{F}_3$} &       & \multicolumn{2}{c}{$\mathcal{F}_4$} \\
\cmidrule{4-5}\cmidrule{7-8}\cmidrule{10-11}\cmidrule{13-14}          & (Boland) &       & Time & \% $D$-Gap &       & Time & \% $D$-Gap &       & Time & \% $D$-Gap &       & Time & \% $D$-Gap \\
    \midrule
    \textit{Source-Based} & 19\%  &       & 3.37  & 7.18\% &       & 3.77  & 7.18\% &       & \textbf{3.06}  & 7.18\% &       & \textbf{1.91}  & 7.18\% \\
    \textit{Terminal-Based} & 19\%  &       & \textbf{2.53}  & 7.18\% &       & \textbf{2.25}  & 7.18\% &       & 25.24  & 7.18\% &       & 5.16  & 7.18\% \\
    \bottomrule
    \end{tabular}}%
  \label{tab:MiningLPrelaxnoBound}%
\end{table}%

Table \ref{tab:MiningLPrelaxnoBound} shows the results of the LP relaxations of the {\it Source-Based} and {\it Terminal-Based} formulations without performing the bound tightening method (additional details in Appendix, Tables \ref{tab:MiningSBLPrelaxnoBound} and \ref{tab:MiningTBLPrelaxnoBound}). The results show that without performing the bound improvement methods, all the LP relaxations of the two multi-commodity flow formulations have the same performance and give identical bounds. These bounds are significantly stronger than those reported by \cite{boland2015special}. This may be true for the mining problem as a special case of the generalized pooling problem in which the supply and the demand are placed at different points of time.

We have followed the steps defined in Section \ref{sec:miningTopoolingInstruction} to improve the bounds of different nodes and arcs in the mining problem. These steps are cheap and simple, and since they do not require solving optimization problems, they have negligible preprocessing time. Thus, unlike OBBT, we do not report the preprocessing time in this case. 

\begin{table}[H]
  \centering
  \caption{Mining Instances \textbf{with Bounds} (LP Relaxations)}
  \scalebox{0.8}{
    \begin{tabular}{cccccccccccccc}
    \toprule
    \multirow{2}[4]{*}{Formulation} & $D$-Gap   &       & \multicolumn{2}{c}{$\mathcal{F}_1$} &       & \multicolumn{2}{c}{$\mathcal{F}_2$} &       & \multicolumn{2}{c}{$\mathcal{F}_3$} &       & \multicolumn{2}{c}{$\mathcal{F}_4$} \\
\cmidrule{4-5}\cmidrule{7-8}\cmidrule{10-11}\cmidrule{13-14}          & (Boland) &       & Time & \% $D$-Gap &       & Time & \% $D$-Gap &       & Time & \% $D$-Gap &       & Time & \% $D$-Gap \\
    \midrule
    \textit{Source-Based} & 19\%  &       & 3.55  & 5.12\% &       & 3.92  & \textbf{4.98\%} &       & 3.99  & 4.01\% &       & 2.75  & 3.94\% \\
    \textit{Terminal-Based} & 19\%  &       & \textbf{2.14}  & \textbf{3.38\%} &       & \textbf{2.60}  & 6.32\% &       & \textbf{2.40}  & \textbf{3.05\%} &       & 2.92  & \textbf{2.98\%} \\
    \bottomrule
    \end{tabular}}%
  \label{tab:MiningLPrelaxBound}%
\end{table}%

Table \ref{tab:MiningLPrelaxBound} indicates that performing the bound tightening method improves the quality of the LP relaxations (detailed results in Tables \ref{tab:MiningSBLPrelaxBound} and \ref{tab:MiningTBLPrelaxBound}). As we can see, $\mathcal{F}_3^\mathcal{S}$ and $\mathcal{F}_4^\mathcal{S}$ give better dual bounds. Recall that our proposed LP relaxation $\mathcal{F}_4^\mathcal{S}$ considers bounds on the row-sum and the column-sum of the decomposed flow variables matrices of the pools simultaneously. As we can see from the table, this relaxation outperforms the others and gives stronger dual bounds.

In addition, the duality gap of $\mathcal{F}_3^\mathcal{T}$ and $\mathcal{F}_4^\mathcal{T}$ are $1\%$ better than those of the {\it Source-Based} formulation respectively. Therefore, we can say that in both cases of using the updated bounds and without them, the {\it row-column} relaxation is the best choice to obtain the dual bounds of the mining instances since it is at least as good as the others and gives better dual bounds while using the updated bounds.

\subsubsection{Discretization Relaxations}
We evaluate the performance of the discretization methods to obtain outer approximations of the mining problems in this section. Table \ref{tab:miningMIPrelax} shows the results of different MIP relaxations, which discretize the variable $q$ at the discretization level $H=3$ for the mining problems (detailed results in Tables \ref{tab:MiningMIPRelaxnoBound} and \ref{tab:MiningMIPRelaxBound}). As discussed previously, these MIP methods are generally stronger than LP relaxations, but they need much more time to give high-quality solutions. The results confirm this fact, and we can see that the running time of the discretization relaxations for the mining problems is not as short as those of the LP relaxations, but the duality gap they give is better. To calculate this duality gap, similar to the LP relaxations, we have considered the best primal bounds reported by \cite{boland2015special}. Meanwhile, without the updated bounds, $\mathcal{M}_1^\mathcal{S}(H)$ performs better than the others in terms of the solution quality and average duality gap.

\begin{table}[H]
  \centering
  \caption{Mining Instances (MIP Relaxations ($H=3$))}
  \scalebox{0.9}{
    \begin{tabular}{cccccccccccc}
    \toprule
    \multirow{2}[4]{*}{Bounds} & \multicolumn{2}{c}{$\mathcal{M}_1^\mathcal{S}(H)$} &       & \multicolumn{2}{c}{$\mathcal{M}_2^\mathcal{S}(H)$} &       & \multicolumn{2}{c}{$\mathcal{M}_1^\mathcal{T}(H)$} &       & \multicolumn{2}{c}{$\mathcal{M}_2^\mathcal{T}(H)$} \\
\cmidrule{2-3}\cmidrule{5-6}\cmidrule{8-9}\cmidrule{11-12}          & Time & \% $D$-Gap &       & Time & \% $D$-Gap &       & Time & \% $D$-Gap &       & Time & \% $D$-Gap \\
    \midrule
    \textit{No} & 1162.50  & 2.15\% &       & 1211.92  & 3.73\% &       & 1889.89  & 3.33\% &       & 1010.25  & 2.53\% \\
    \textit{Yes} & 1167.92  & \textbf{1.41\%} &       & 1339.23  & \textbf{2.48\%} &       & \textbf{1064.11}  & \textbf{1.06\%} &       & 1141.55  & \textbf{2.18\%} \\
    \bottomrule
    \end{tabular}}%
  \label{tab:miningMIPrelax}%
\end{table}%

The results of the MIP relaxations in conjunction with the bound tightening indicate that improving the bounds of arcs and nodes of the mining problem has a positive impact on the dual bound obtained by the discretization relaxations. The discretization method $\mathcal{M}_1^\mathcal{T}(H)$ has not only improved the duality gap significantly but also the running time is much less than the case of having no updated bounds. The rest of the methods have improved the relaxation quality by making use of the bound tightening method within almost the same amount of average running time. For some instances, the MIP relaxations are running out of the time limit of one hour, which is the cause of the large average of time needed to obtain a high-quality dual bound.

\subsubsection{Valid Inequalities}
\label{sec:comp-valid}
We developed some valid inequalities in Section \ref{sec:validIneq}, which have the lower bounds of the pools as the denominator of a fraction. These lower bounds exist in the mining problem, and we can improve them. However, for the literature instances, they do not exist generally. Therefore, we only evaluate the new valid inequalities' performance with the mining instances. In this section, we aim to evaluate the performance of adding the valid inequalities to the {\it row-wise}, {\it column-wise}, and their intersection as well as the {\it row-column} relaxations of the {\it Source-Based} and {\it Terminal-Based} formulations separately. We report the results of the valid inequalities $\mathcal{V}_{ab}$ and $\mathcal{V}_{ac}$ since they have reasonable performance in the mining instances.

\begin{table}[H]
  \centering
  \caption{Average Running Time and Duality Gap with Bound Tightening ({\it Source-Based}: LP+Valid Inequalities)}
  \scalebox{0.9}{
    \begin{tabular}{ccccccccccccc}
    \toprule
    \multirow{2}[4]{*}{Valid Ineq.} &       & \multicolumn{2}{c}{$\mathcal{F}_1^\mathcal{S}$} &       & \multicolumn{2}{c}{$\mathcal{F}_2^\mathcal{S}$} &       & \multicolumn{2}{c}{$\mathcal{F}_3^\mathcal{S}$} &       & \multicolumn{2}{c}{$\mathcal{F}_4^\mathcal{S}$} \\
\cmidrule{3-4}\cmidrule{6-7}\cmidrule{9-10}\cmidrule{12-13}          &       & Time & \% $D$-Gap &       & Time & \% $D$-Gap &       & Time & \% $D$-Gap &       & Time & \% $D$-Gap \\
    \midrule
    -     &       & 3.55  & 5.12\% &       & 3.92  & 4.98\% &       & 3.99  & 4.01\% &       & 2.75  & 3.94\% \\
    $\mathcal{V}_{ab}^\mathcal{S}$  &       & 2.60  & \textbf{4.70\%} &       & 2.51  & \textbf{4.54\%} &       & 2.79  & \textbf{3.92\%} &       & 4.32  & \textbf{3.88\%} \\
    $\mathcal{V}_{ac}^\mathcal{S}$  &       & 2.22  & 5.04\% &       & 2.40  & 4.97\% &       & 2.17  & 4.00\% &       & 4.10  & 3.93\% \\
    \bottomrule
    \end{tabular}}%
  \label{tab:summiningSBLP+Valid}%
\end{table}%

Table \ref{tab:summiningSBLP+Valid} shows the running time and the duality gap of adding the valid inequalities to the LP relaxations of the \textit{Source-Based} formulation (additional details in Appendix, Tables \ref{tab:MiningSBLP+6a6bBound} and \ref{tab:MiningSBLP+6a6cBound}). We have used the updated bounds of the arcs and nodes obtained by the proposed bound-tightening method. We observe that the best performance of the relaxations is achieved while adding the $\mathcal{V}_{ab}^{\mathcal{S}}$. However, adding the $\mathcal{V}_{ac}^{\mathcal{S}}$ has a positive effect on improving the duality gap of the {\it Source-Based} formulation. Furthermore, as we expected, the {\it row-column} relaxation also gives the highest quality dual bounds in this case.

To obtain high-quality dual bounds for the generalized pooling problem counterpart of the mining problem instances, we can make use of this addition as it can yield less duality gap than the LP relaxations in the same average amount of running time.

\begin{table}[H]
  \centering
  \caption{Average Running Time and Duality Gap with Bound Tightening ({\it Terminal-Based}: LP+Valid Inequalities)}
  \scalebox{0.9}{
    \begin{tabular}{ccccccccccccc}
    \toprule
    \multirow{2}[4]{*}{Valid Ineq.} &       & \multicolumn{2}{c}{$\mathcal{F}_1^\mathcal{T}$} &       & \multicolumn{2}{c}{$\mathcal{F}_2^\mathcal{T}$} &       & \multicolumn{2}{c}{$\mathcal{F}_3^\mathcal{T}$} &       & \multicolumn{2}{c}{$\mathcal{F}_4^\mathcal{T}$} \\
\cmidrule{3-4}\cmidrule{6-7}\cmidrule{9-10}\cmidrule{12-13}          &       & Time & \% $D$-Gap &       & Time & \% $D$-Gap &       & Time & \% $D$-Gap &       & Time & \% $D$-Gap \\
    \midrule
    -     &       & 2.14  & 3.38\% &       & 2.60  & 6.32\% &       & 2.40  & 3.05\% &       & 2.92  & 2.98\% \\
    $\mathcal{V}_{ab}^\mathcal{T}$  &       & 3.43  & \textbf{3.26\%} &       & 3.00  & \textbf{4.38\%} &       & 4.18  & \textbf{3.03\%} &       & 6.70  & \textbf{2.96\%} \\
    $\mathcal{V}_{ac}^\mathcal{T}$  &       & 1.81  & 3.38\% &       & 2.10  & 6.32\% &       & 1.54  & 3.05\% &       & 4.45  & 2.98\% \\
    \bottomrule
    \end{tabular}}%
  \label{tab:summiningTBLP+Valid}%
\end{table}%

Table \ref{tab:summiningTBLP+Valid} summarizes the results of LP relaxations of the {\it Terminal-Based} formulation while we add valid inequalities to them (additional details in Appendix, Tables \ref{tab:MiningTBLP+6a6bBound} and \ref{tab:MiningTBLP+6a6cBound}). As shown in the table, the best dual bound of the relaxations is obtained in addition to $\mathcal{V}_{ab}^{\mathcal{T}}$. For the {\it Terminal-Based} formulation, adding $\mathcal{V}_{ac}^{\mathcal{T}}$ does not improve the dual bound quality, and it remains the same as the case without any added valid inequalities. In three different cases shown in the table, the best dual bounds are obtained while using the {\it row-column} relaxation.

\section{Conclusion}\label{sec:conclusion}

{
In this paper, we focused on the pooling problem, a challenging nonlinear and nonconvex network flow problem. Our analysis focused on a recently proposed rank-one-based formulation of the problem. Firstly, we proved that the convex hull of a recurring substructure in the formulation defined as the set of nonnegative, rank-one matrices with bounded row sums, column sums, and the overall sum is second-other cone representable. Secondly, we introduced novel linear programming relaxations based on this analysis, which outperform existing approaches. Thirdly, we derived   valid inequalities using the Reformulation Linearization Technique to further  strengthen the dual bounds. Finally, to improve the bounds on node and arc capacities, we utilized Optimization-Based Bound Tightening for generic problem instances, and a simple and cost-effective bound-tightening method tailored for time-indexed pooling problem instances. Computational experiments on some benchmark instances showed that our approach has the potential of producing accurate and efficient results thanks to the improved formulations and bound tightening techniques.
}

\section*{Acknowledgments}

The authors thank The Scientific and Technological Research Council of Turkey (TÜBİTAK) for supporting this work (project number: 119M855).

\section*{Statement of Conflict of Interest}
The authors declare that they have no known competing financial interests or personal relationships that could have appeared to influence the work reported in this paper.

\bibliographystyle{plainnat}
\bibliography{Rankone-bib}

\appendix

\section{Computations}\label{App:sec:computations}

\subsection{Literature Instances}

\begin{table}[H]
  \centering
  \caption{Characteristics of the literature instances}
\begin{tabular}{cccccc}
    \toprule
    Instance & $|S|$   & $|I|$   & $|T|$   & $|A|$   & $|K|$ \\
    \midrule
    Haverly1 & 3     & 2     & 2     & 9     & 1 \\
    Haverly2 & 3     & 2     & 2     & 9     & 1 \\
    Haverly3 & 3     & 2     & 2     & 9     & 1 \\
    BenTal4 & 4     & 2     & 2     & 10    & 1 \\
    BenTal5 & 5     & 3     & 5     & 38    & 2 \\
    Adhya1 & 5     & 2     & 4     & 15    & 4 \\
    Adhya2 & 5     & 2     & 4     & 15    & 6 \\
    Adhya3 & 8     & 3     & 4     & 26    & 6 \\
    Adhya4 & 8     & 2     & 5     & 20    & 4 \\
    Foulds2 & 6     & 2     & 4     & 22    & 1 \\
    Foulds3 & 11    & 8     & 16    & 216   & 1 \\
    Foulds4 & 11    & 8     & 16    & 216   & 1 \\
    Foulds5 & 11    & 4     & 16    & 108   & 1 \\
    \bottomrule
    \end{tabular}%
  \label{tab:litInsChar}%
\end{table}%

\subsubsection{Exact Formulations}

\begin{table}[H]
  \centering
  \caption{Literature Instances ({\it Source-Based}: Exact).}
  \scalebox{0.9}{
    \begin{tabular}{ccccccccc}
    \toprule
    \multirow{2}[4]{*}{Instance} & \multicolumn{3}{c}{Gurobi} &       & \multicolumn{4}{c}{Gurobi with OBBT} \\
\cmidrule{2-4}\cmidrule{6-9}          & Obj. Value* & Time & \% $O$-Gap &       & Obj. Value & Prep. Time** & Gurobi Time & \% $O$-Gap \\
    \midrule
    Haverly1 & -400  & 0.14  & 0.00\% &       & -400  & 0.21  & 0.02  & 0.00\% \\
    Haverly2 & -600  & 0.24  & 0.00\% &       & -600  & 0.24  & 0.05  & 0.00\% \\
    Haverly3 & -750  & 14.76 & 0.01\% &       & -750  & 0.29  & 0.08  & 0.00\% \\
    BenTal4 & -450  & 0.30  & 0.00\% &       & -450  & 0.32  & 0.02  & 0.00\% \\
    BenTal5 & -3500 & 6.59  & 0.00\% &       & -3500 & 1.53  & 1.12  & 0.00\% \\
    Adhya1 & -550  & 3600.00 & 2.03\% &       & -550  & 0.52  & 0.45  & 0.00\% \\
    Adhya2 & -550  & 3600.02 & 0.07\% &       & -550  & 0.75  & 0.23  & 0.00\% \\
    Adhya3 & -560  & 3600.01 & 2.48\% &       & -561  & 1.16  & 3.66  & 0.00\% \\
    Adhya4 & -878  & 3600.16 & 4.04\% &       & -878  & 0.87  & 0.28  & 0.01\% \\
    Foulds2 & -1100 & 0.45  & 0.00\% &       & -1100 & 0.92  & 0.20  & 0.00\% \\
    Foulds3 & -8    & 20.78 & 0.00\% &       & -8    & 104.98 & 1867.48 & 0.00\% \\
    Foulds4 & -8    & 66.73 & 0.00\% &       & -8    & 88.67 & 432.41 & 0.00\% \\
    Foulds5 & -8    & 12.07 & 0.00\% &       & -8    & 41.65 & 104.89 & 0.00\% \\
    \textbf{Average} &       & \textbf{1117.10} & \textbf{0.66\%} &       &       & \textbf{18.62} & \textbf{185.45} & \textbf{0.00\%} \\
    \midrule
    \multicolumn{9}{c}{* Objective Value. ** Preprocessing time.} \\
    \end{tabular}}%
  \label{tab:litSBexact}%
\end{table}%

\begin{table}[H]
  \centering
  \caption{Literature Instances ({\it Terminal-Based}: Exact).}
  \scalebox{0.9}{
    \begin{tabular}{ccccccccc}
    \toprule
    \multirow{2}[4]{*}{Instance} & \multicolumn{3}{c}{Gurobi} &       & \multicolumn{4}{c}{Gurobi+OBBT} \\
\cmidrule{2-4}\cmidrule{6-9}          & Obj. Value & Time  & \% $O$-Gap &       & Obj. Value & Prep. Time & Gurobi Time & \% $O$-Gap \\
    \midrule
    Haverly1 & -400  & 0.03  & 0.01\% &       & -400  & 0.21  & 0.02  & 0.00\% \\
    Haverly2 & -600  & 0.03  & 0.00\% &       & -600  & 0.24  & 0.02  & 0.00\% \\
    Haverly3 & -750  & 0.03  & 0.00\% &       & -750  & 0.29  & 0.00  & 0.00\% \\
    BenTal4 & -450  & 0.03  & 0.00\% &       & -450  & 0.32  & 0.02  & 0.00\% \\
    BenTal5 & -3500 & 0.22  & 0.00\% &       & -3500 & 1.53  & 0.08  & 0.00\% \\
    Adhya1 & -550  & 3600.02 & 8.67\% &       & -550  & 0.52  & 0.09  & 0.00\% \\
    Adhya2 & -550  & 0.25  & 0.00\% &       & -550  & 0.75  & 0.11  & 0.00\% \\
    Adhya3 & -561  & 0.90  & 0.00\% &       & -561  & 1.16  & 0.14  & 0.00\% \\
    Adhya4 & -878  & 0.39  & 0.00\% &       & -878  & 0.87  & 0.12  & 0.00\% \\
    Foulds2 & -1100 & 0.03  & 0.00\% &       & -1100 & 0.92  & 0.02  & 0.00\% \\
    Foulds3 & -8    & 1.68  & 0.00\% &       & -8    & 104.98 & 0.86  & 0.00\% \\
    Foulds4 & -8    & 0.77  & 0.00\% &       & -8    & 88.67 & 0.39  & 0.00\% \\
    Foulds5 & -8    & 2.15  & 0.00\% &       & -8    & 41.65 & 32.53 & 0.00\% \\
    \textbf{Average} &       & \textbf{277.43} & \textbf{0.67\%} &       &       & \textbf{18.62} & \textbf{2.65} & \textbf{0.00\%} \\
    \bottomrule
    \end{tabular}}%
  \label{tab:litTBexact}%
\end{table}%

\subsubsection{Linear Programming Relaxations}

\begin{table}[H]
  \centering
  \caption{Literature Instances Without OBBT ({\it Source-Based}: LP Relaxations)}
  \scalebox{0.9}{
    \begin{tabular}{cccccccccccc}
    \toprule
    \multirow{2}[4]{*}{Instance} & \multicolumn{2}{c}{$\mathcal{F}_1^\mathcal{S}$} &       & \multicolumn{2}{c}{$\mathcal{F}_2^\mathcal{S}$} &       & \multicolumn{2}{c}{$\mathcal{F}_3^\mathcal{S}$} &       & \multicolumn{2}{c}{$\mathcal{F}_4^\mathcal{S}$} \\
\cmidrule{2-3}\cmidrule{5-6}\cmidrule{8-9}\cmidrule{11-12}          & Time & \% $D$-Gap &       & Time & \% $D$-Gap &       & Time & \% $D$-Gap &       & Time & \% $D$-Gap \\
    \midrule
    Haverly1 & 0.00  & 25.00\% &       & 0.02  & 25.00\% &       & 0.02  & 25.00\% &       & 0.02  & 25.00\% \\
    Haverly2 & 0.02  & 66.67\% &       & 0.02  & 66.67\% &       & 0.02  & 66.67\% &       & 0.00  & 66.67\% \\
    Haverly3 & 0.02  & 16.67\% &       & 0.02  & 16.67\% &       & 0.02  & 16.67\% &       & 0.02  & 16.67\% \\
    BenTal4 & 0.01  & 22.22\% &       & 0.00  & 22.22\% &       & 0.00  & 22.18\% &       & 0.00  & 22.18\% \\
    BenTal5 & 0.02  & 0.00\% &       & 0.02  & 0.00\% &       & 0.02  & 0.00\% &       & 0.03  & 0.00\% \\
    Adhya1 & 0.02  & 55.18\% &       & 0.02  & 55.68\% &       & 0.02  & 55.18\% &       & 0.02  & 55.18\% \\
\bottomrule
    \end{tabular}}%
  \label{tab:litSBLPnoOBBT}%
\end{table}%

\begin{table}[H]
  \centering
  \caption{Literature Instances Without OBBT ({\it Terminal-Based}: LP Relaxations)}
  \scalebox{0.9}{
    \begin{tabular}{cccccccccccc}
    \toprule
    \multirow{2}[4]{*}{Instance} & \multicolumn{2}{c}{$\mathcal{F}_1^\mathcal{T}$} &       & \multicolumn{2}{c}{$\mathcal{F}_2^\mathcal{T}$} &       & \multicolumn{2}{c}{$\mathcal{F}_3^\mathcal{T}$} &       & \multicolumn{2}{c}{$\mathcal{F}_4^\mathcal{T}$} \\
\cmidrule{2-3}\cmidrule{5-6}\cmidrule{8-9}\cmidrule{11-12}          & Time & \% $D$-Gap &       & Time & \% $D$-Gap &       & Time & \% $D$-Gap &       & Time & \% $D$-Gap \\
    \midrule
    Haverly1 & 0.00  & 25.00\% &       & 0.00  & 25.00\% &       & 0.00  & 25.00\% &       & 0.02  & 25.00\% \\
    Haverly2 & 0.02  & 66.67\% &       & 0.00  & 66.67\% &       & 0.02  & 66.67\% &       & 0.02  & 66.67\% \\
    Haverly3 & 0.00  & 16.67\% &       & 0.02  & 6.67\% &       & 0.01  & 6.67\% &       & 0.02  & 6.67\% \\
    BenTal4 & 0.02  & 21.30\% &       & 0.02  & 22.22\% &       & 0.00  & 21.30\% &       & 0.02  & 21.30\% \\
    BenTal5 & 0.02  & 0.00\% &       & 0.02  & 0.00\% &       & 0.02  & 0.00\% &       & 0.02  & 0.00\% \\
    Adhya1 & 0.00  & 55.68\% &       & 0.02  & 52.78\% &       & 0.00  & 52.78\% &       & 0.02  & 52.78\% \\
\bottomrule
    \end{tabular}}%
  \label{tab:litTBLPnoOBBT}%
\end{table}%

\begin{table}[H]
  \centering
  \caption{Literature Instances With OBBT ({\it Source-Based}: LP Relaxations)}
  \scalebox{0.9}{
    \begin{tabular}{cccccccccccc}
    \toprule
    \multirow{2}[4]{*}{Instance} & \multicolumn{2}{c}{$\mathcal{F}_1^\mathcal{S}$} &       & \multicolumn{2}{c}{$\mathcal{F}_2^\mathcal{S}$} &       & \multicolumn{2}{c}{$\mathcal{F}_3^\mathcal{S}$} &       & \multicolumn{2}{c}{$\mathcal{F}_4^\mathcal{S}$} \\
\cmidrule{2-3}\cmidrule{5-6}\cmidrule{8-9}\cmidrule{11-12}          & Time & \% $D$-Gap &       & Time & \% $D$-Gap &       & Time & \% $D$-Gap &       & Time & \% $D$-Gap \\
    \midrule
    Haverly1 & 0.00  & 0.00\% &       & 0.02  & 0.00\% &       & 0.00  & 0.00\% &       & 0.00  & 0.00\% \\
    Haverly2 & 0.02  & 37.51\% &       & 0.00  & 0.00\% &       & 0.00  & 0.00\% &       & 0.02  & 0.00\% \\
    Haverly3 & 0.02  & 4.86\% &       & 0.00  & 0.00\% &       & 0.02  & 0.00\% &       & 0.02  & 0.00\% \\
    BenTal4 & 0.02  & 0.00\% &       & 0.02  & 0.00\% &       & 0.02  & 0.00\% &       & 0.02  & 0.00\% \\
    BenTal5 & 0.02  & 0.00\% &       & 0.02  & 0.00\% &       & 0.02  & 0.00\% &       & 0.02  & 0.00\% \\
    Adhya1 & 0.00  & 47.79\% &       & 0.02  & 49.01\% &       & 0.02  & 43.90\% &       & 0.00  & 43.90\% \\
\bottomrule
    \end{tabular}}%
  \label{tab:litSBLPOBBT}%
\end{table}%

\begin{table}[H]
  \centering
  \caption{Literature Instances With OBBT ({\it Terminal-Based}: LP Relaxations)}
  \scalebox{0.9}{
    \begin{tabular}{cccccccccccc}
    \toprule
    \multirow{2}[4]{*}{Instance} & \multicolumn{2}{c}{$\mathcal{F}_1^\mathcal{T}$} &       & \multicolumn{2}{c}{$\mathcal{F}_2^\mathcal{T}$} &       & \multicolumn{2}{c}{$\mathcal{F}_3^\mathcal{T}$} &       & \multicolumn{2}{c}{$\mathcal{F}_4^\mathcal{T}$} \\
\cmidrule{2-3}\cmidrule{5-6}\cmidrule{8-9}\cmidrule{11-12}          & Time & \% $D$-Gap &       & Time & \% $D$-Gap &       & Time & \% $D$-Gap &       & Time & \% $D$-Gap \\
    \midrule
    Haverly1 & 0.02  & 0.00\% &       & 0.00  & 0.00\% &       & 0.01  & 0.00\% &       & 0.02  & 0.00\% \\
    Haverly2 & 0.02  & 0.00\% &       & 0.02  & 37.50\% &       & 0.00  & 0.00\% &       & 0.00  & 0.00\% \\
    Haverly3 & 0.00  & 0.00\% &       & 0.00  & 4.86\% &       & 0.02  & 0.00\% &       & 0.02  & 0.00\% \\
    BenTal4 & 0.02  & 0.00\% &       & 0.00  & 0.00\% &       & 0.00  & 0.00\% &       & 0.00  & 0.00\% \\
    BenTal5 & 0.00  & 0.00\% &       & 0.02  & 0.00\% &       & 0.00  & 0.00\% &       & 0.02  & 0.00\% \\
    Adhya1 & 0.00  & 49.01\% &       & 0.00  & 47.79\% &       & 0.02  & 43.90\% &       & 0.02  & 43.90\% \\
\bottomrule
    \end{tabular}}%
  \label{tab:litTBLPOBBT}%
\end{table}%

\subsubsection{Discretization Relaxations}

\begin{table}[H]
  \centering
  \caption{Literature Instances Without OBBT (MIP Relaxations: $H=3$).}
  \scalebox{0.9}{
    \begin{tabular}{cccccccccccc}
    \toprule
    \multirow{2}[4]{*}{Instance} & \multicolumn{2}{c}{$\mathcal{M}_1^\mathcal{S}(H)$} &       & \multicolumn{2}{c}{$\mathcal{M}_2^\mathcal{S}(H)$} &       & \multicolumn{2}{c}{$\mathcal{M}_1^\mathcal{T}(H)$} &       & \multicolumn{2}{c}{$\mathcal{M}_2^\mathcal{T}(H)$} \\
\cmidrule{2-3}\cmidrule{5-6}\cmidrule{8-9}\cmidrule{11-12}          & Time & \% $D$-Gap &       & Time & \% $D$-Gap &       & Time & \% $D$-Gap &       & Time & \% $D$-Gap \\
    \midrule
    Haverly1 & 0.14  & 6.45\% &       & 0.08  & 0.00\% &       & 0.02  & 0.00\% &       & 0.02  & 0.00\% \\
    Haverly2 & 0.11  & 34.44\% &       & 0.09  & 0.00\% &       & 0.02  & 0.00\% &       & 0.02  & 0.00\% \\
    Haverly3 & 0.13  & 11.90\% &       & 0.09  & 0.00\% &       & 0.02  & 1.90\% &       & 0.02  & 0.00\% \\
    BenTal4 & 0.14  & 0.00\% &       & 0.11  & 0.00\% &       & 0.02  & 0.00\% &       & 0.02  & 0.00\% \\
    BenTal5 & 0.30  & 0.00\% &       & 2.79  & 0.00\% &       & 0.06  & 0.00\% &       & 0.05  & 0.00\% \\
    Adhya1 & 0.33  & 37.60\% &       & 0.39  & 3.05\% &       & 0.09  & 2.66\% &       & 0.11  & 0.80\% \\
\bottomrule
    \end{tabular}}%
  \label{tab:litMIPrelaxnoOBBT}%
\end{table}%

\begin{table}[H]
  \centering
  \caption{Literature Instances With OBBT (MIP Relaxations: $H=3$).}
  \scalebox{0.9}{
    \begin{tabular}{cccccccccccc}
    \toprule
    \multirow{2}[4]{*}{Instance} & \multicolumn{2}{c}{$\mathcal{M}_1^\mathcal{S}(H)$} &       & \multicolumn{2}{c}{$\mathcal{M}_2^\mathcal{S}(H)$} &       & \multicolumn{2}{c}{$\mathcal{M}_1^\mathcal{T}(H)$} &       & \multicolumn{2}{c}{$\mathcal{M}_2^\mathcal{T}(H)$} \\
\cmidrule{2-3}\cmidrule{5-6}\cmidrule{8-9}\cmidrule{11-12}          & Time & \% $D$-Gap &       & Time & \% $D$-Gap &       & Time & \% $D$-Gap &       & Time & \% $D$-Gap \\
    \midrule
    Haverly1 & 0.02  & 0.00\% &       & 0.02  & 0.00\% &       & 0.02  & 0.00\% &       & 0.02  & 0.00\% \\
    Haverly2 & 0.02  & 0.00\% &       & 0.02  & 0.00\% &       & 0.02  & 0.00\% &       & 0.02  & 0.00\% \\
    Haverly3 & 0.02  & 0.00\% &       & 0.02  & 0.00\% &       & 0.02  & 0.00\% &       & 0.02  & 0.00\% \\
    BenTal4 & 0.02  & 0.00\% &       & 0.03  & 0.00\% &       & 0.02  & 0.00\% &       & 0.02  & 0.00\% \\
    BenTal5 & 0.08  & 0.00\% &       & 0.30  & 0.00\% &       & 0.12  & 0.00\% &       & 0.06  & 0.00\% \\
    Adhya1 & 0.08  & 2.03\% &       & 0.12  & 0.66\% &       & 0.08  & 2.03\% &       & 0.05  & 0.66\% \\
\bottomrule
    \end{tabular}}%
  \label{tab:litMIPrelaxOBBT}%
\end{table}%

\subsubsection{Discretization Restrictions}

\begin{table}[H]
  \centering
  \caption{Literature Instances Without OBBT (MIP Restrictions: $H=3$).}
  \scalebox{0.9}{
    \begin{tabular}{cccccccccccc}
    \toprule
    \multirow{2}[4]{*}{Instance} & \multicolumn{2}{c}{$\mathcal{G}_1^\mathcal{S}(H)$} &       & \multicolumn{2}{c}{$\mathcal{G}_2^\mathcal{S}(H)$} &       & \multicolumn{2}{c}{$\mathcal{G}_1^\mathcal{T}(H)$} &       & \multicolumn{2}{c}{$\mathcal{G}_2^\mathcal{T}(H)$} \\
\cmidrule{2-3}\cmidrule{5-6}\cmidrule{8-9}\cmidrule{11-12}          & Time & \% $P$-Gap &       & Time & \% $P$-Gap &       & Time & \% $P$-Gap &       & Time & \% $P$-Gap \\

    \midrule
    Haverly1 & 0.08  & 0.00\% &       & 0.05  & 0.00\% &       & 0.02  & 0.00\% &       & 0.02  & 0.00\% \\
    Haverly2 & 0.09  & 0.00\% &       & 0.08  & 0.00\% &       & 0.02  & 0.00\% &       & 0.03  & 0.00\% \\
    Haverly3 & 0.09  & 0.00\% &       & 0.06  & 3.45\% &       & 0.03  & 0.00\% &       & 0.03  & 4.17\% \\
    BenTal4 & 0.10  & 0.00\% &       & 0.09  & 0.00\% &       & 0.03  & 0.00\% &       & 0.03  & 1.61\% \\
    BenTal5 & 3.86  & 0.00\% &       & 0.67  & 0.00\% &       & 1.21  & 0.00\% &       & 0.06  & 0.00\% \\
    Adhya1 & 0.55  & 0.08\% &       & 0.28  & 2.42\% &       & 0.08  & 0.08\% &       & 0.06  & 2.42\% \\
\bottomrule
    \end{tabular}}%
  \label{tab:litMIPrestrictNoOBBT}%
\end{table}%

\begin{table}[H]
  \centering
  \caption{Literature Instances With OBBT (MIP Restrictions: $H=3$).}
  \scalebox{0.9}{
    \begin{tabular}{cccccccccccc}
    \toprule
   \multirow{2}[4]{*}{Instance} & \multicolumn{2}{c}{$\mathcal{G}_1^\mathcal{S}(H)$} &       & \multicolumn{2}{c}{$\mathcal{G}_2^\mathcal{S}(H)$} &       & \multicolumn{2}{c}{$\mathcal{G}_1^\mathcal{T}(H)$} &       & \multicolumn{2}{c}{$\mathcal{G}_2^\mathcal{T}(H)$} \\
\cmidrule{2-3}\cmidrule{5-6}\cmidrule{8-9}\cmidrule{11-12}          & Time & \% $P$-Gap &       & Time & \% $P$-Gap &       & Time & \% $P$-Gap &       & Time & \% $P$-Gap \\
    \midrule
    Haverly1 & 0.02  & 0.00\% &       & 0.02  & 0.00\% &       & 0.02  & 0.00\% &       & 0.00  & 0.00\% \\
    Haverly2 & 0.00  & 0.00\% &       & 0.02  & 0.00\% &       & 0.02  & 0.00\% &       & 0.02  & 0.00\% \\
    Haverly3 & 0.02  & 0.00\% &       & 0.03  & 15.36\% &       & 0.00  & 0.00\% &       & 0.02  & 15.36\% \\
    BenTal4 & 0.02  & 0.00\% &       & 0.02  & 1.61\% &       & 0.02  & 0.00\% &       & 0.02  & 1.61\% \\
    BenTal5 & 0.25  & 0.00\% &       & 0.16  & 0.00\% &       & 0.03  & 0.00\% &       & 0.05  & 0.00\% \\
    Adhya1 & 0.05  & 0.05\% &       & 0.06  & 2.38\% &       & 0.06  & 0.05\% &       & 0.03  & 2.38\% \\
\bottomrule
    \end{tabular}}%
  \label{tab:litMIPrestrictOBBT}%
\end{table}%

\subsection{Mining Instances}
\begin{table}[H]
  \centering
  \caption{Characteristics of the mining instances}
  \scalebox{0.9}{
    \begin{tabular}{cccccccccccc}
    \toprule
    Instance & $|S|$   & $|I|$   & $|T|$   & $|A|$   & $|ASI|$ & $|AII|$ & $|AIT|$ & $|K|$   & $SP 1$  & $SP 2$ \\
    \midrule
    2009H2 & 73    & 73    & 49    & 242   & 73    & 71    & 98    & 4     & 37    & 36 \\
    2009Q3 & 31    & 31    & 21    & 102   & 31    & 29    & 42    & 4     & 16    & 15 \\
    2009Q4 & 38    & 38    & 26    & 126   & 38    & 36    & 52    & 4     & 19    & 19 \\
    2010  & 170   & 170   & 122   & 582   & 170   & 168   & 244   & 4     & 82    & 88 \\
    2010H1 & 86    & 86    & 63    & 296   & 86    & 84    & 126   & 4     & 41    & 45 \\
    2010H2 & 84    & 84    & 58    & 282   & 84    & 82    & 116   & 4     & 41    & 43 \\
    2010Q1 & 39    & 39    & 28    & 132   & 39    & 37    & 56    & 4     & 19    & 20 \\
\bottomrule
    \end{tabular}}%
  \label{tab:MiningInsChar}%
\end{table}%

\subsubsection{Exact Formulations}
\begin{table}[H]
  \centering
  \caption{Mining Instances without Bound Tightening (Exact)}
  \scalebox{0.7}{
    \begin{tabular}{cccccccc}
    \toprule
    \multirow{2}[4]{*}{Instance} & \multicolumn{3}{c}{Source-Based} &       & \multicolumn{3}{c}{Terminal-Based} \\
\cmidrule{2-4}\cmidrule{6-8}          & Obj. Value & Time  & \% $O$-Gap &       & Obj. Value & Time  & \% $O$-Gap \\
    \midrule
    2009H2 & 4151473 & 3600.37 & 3.57\% &       & 4144584 & 3600.50 & 1.18\% \\
    2009Q3 & 2281115 & 99.05 & 0.00\% &       & 2281116 & 8.95  & 0.01\% \\
    2009Q4 & 1787783 & 3600.10 & 13.70\% &       & 1786489 & 3600.21 & 5.80\% \\
    2010  & -   & - & -     &       & 12520846 & 3600.62 & 13.81\% \\
    2010H1 & 8179534 & 3600.66 & 24.63\% &       & 7268601 & 3600.52 & 4.32\% \\
    2010H2 & 4251752 & 3600.88 & 2.78\% &       & 4242371 & 3600.53 & 1.29\% \\
\bottomrule
    
    \end{tabular}}%
  \label{tab:MiningExactnoBound}%
\end{table}%

\begin{table}[H]
  \centering
  \caption{Mining Instances with Bound Tightening (Exact)}
  \scalebox{0.7}{
    \begin{tabular}{cccccccc}
    \toprule
    \multirow{2}[4]{*}{Instance} & \multicolumn{3}{c}{Source-Based} &       & \multicolumn{3}{c}{Terminal-Based} \\
\cmidrule{2-4}\cmidrule{6-8}          & Obj. Value & Time  & \% $O$-Gap &       & Obj. Value & Time  & \% $O$-Gap \\
    \midrule
    2009H2 & 4162660 & 3600.57 & 3.52\% &       & 4143629 & 3600.28 & 0.69\% \\
    2009Q3 & 2281127 & 56.51 & 0.00\% &       & 2281165 & 3.77  & 0.00\% \\
    2009Q4 & 1779013 & 3600.28 & 9.06\% &       & 1805388 & 3600.31 & 6.92\% \\
    2010  & -     & -     & -     &       & 12033027 & 3600.57 & 7.60\% \\
    2010H1 & 8742718 & 3600.61 & 27.82\% &       & 7219261 & 3600.25 & 3.64\% \\
    2010H2 & 4251731 & 3600.97 & 2.65\% &       & 4242420 & 3600.31 & 1.04\% \\
\bottomrule
    
    \end{tabular}}%
  \label{tab:MiningExactBound}%
\end{table}%

\subsubsection{Linear Programming Relaxations}
 
\begin{table}[H]
  \centering
  \caption{Mining Instances Without Bound Tightening ({\it Source-Based}: LP Relaxations)}
  \scalebox{0.7}{
    \begin{tabular}{cccccccccccccc}
    \toprule
    \multirow{2}[4]{*}{Instance} & $D$-Gap   &       & \multicolumn{2}{c}{$\mathcal{F}_1^\mathcal{S}$} &       & \multicolumn{2}{c}{$\mathcal{F}_2^\mathcal{S}$} &       & \multicolumn{2}{c}{$\mathcal{F}_3^\mathcal{S}$} &       & \multicolumn{2}{c}{$\mathcal{F}_4^\mathcal{S}$} \\
\cmidrule{4-5}\cmidrule{7-8}\cmidrule{10-11}\cmidrule{13-14}          & (Boland et al.) &       & Time & \% $D$-Gap &       & Time & \% $D$-Gap &       & Time & \% $D$-Gap &       & Time & \% $D$-Gap \\
    \midrule
    2009H2 & 37\%  &       & 0.93  & 8.70\% &       & 0.94  & 8.70\% &       & 1.28  & 8.70\% &       & 1.76  & 8.70\% \\
    2009Q3 & 29\%  &       & 0.08  & 4.18\% &       & 0.08  & 4.18\% &       & 0.08  & 4.18\% &       & 0.16  & 4.18\% \\
    2009Q4 & 42\%  &       & 0.11  & 29.70\% &       & 0.11  & 29.70\% &       & 0.17  & 29.70\% &       & 0.17  & 29.70\% \\
    2010  & 26\%  &       & 58.20 & 11.60\% &       & 57.25 & 11.60\% &       & 45.15 & 11.60\% &       & 22.07 & 11.60\% \\
    2010H1 & 32\%  &       & 1.34  & 16.30\% &       & 1.45  & 16.30\% &       & 3.73  & 16.30\% &       & 2.47  & 16.30\% \\
    2010H2 & 15\%  &       & 1.33  & 4.52\% &       & 1.28  & 4.52\% &       & 1.41  & 4.52\% &       & 2.52  & 4.52\% \\
\bottomrule
    \end{tabular}}%
  \label{tab:MiningSBLPrelaxnoBound}%
\end{table}%

\begin{table}[H]
  \centering
   \caption{Mining Instances Without Bound Tightening ({\it Terminal-Based}: LP Relaxations)}
  \scalebox{0.7}{
    \begin{tabular}{cccccccccccccc}
    \toprule
    \multirow{2}[4]{*}{Instance} & $D$-Gap   &       & \multicolumn{2}{c}{$\mathcal{F}_1^\mathcal{T}$} &       & \multicolumn{2}{c}{$\mathcal{F}_2^\mathcal{T}$} &       & \multicolumn{2}{c}{$\mathcal{F}_3^\mathcal{T}$} &       & \multicolumn{2}{c}{$\mathcal{F}_4^\mathcal{T}$} \\
\cmidrule{4-5}\cmidrule{7-8}\cmidrule{10-11}\cmidrule{13-14}          & (Boland et al.) &       & Time & \% $D$-Gap &       & Time & \% $D$-Gap &       & Time & \% $D$-Gap &       & Time & \% $D$-Gap \\
    \midrule
    2009H2 & 37\%  &       & 0.92  & 8.70\% &       & 0.72  & 8.70\% &       & 1.21  & 8.70\% &       & 1.55  & 8.70\% \\
    2009Q3 & 29\%  &       & 0.08  & 4.18\% &       & 0.08  & 4.18\% &       & 0.09  & 4.18\% &       & 0.09  & 4.18\% \\
    2009Q4 & 42\%  &       & 0.11  & 29.70\% &       & 0.08  & 29.70\% &       & 0.19  & 29.70\% &       & 0.35  & 29.70\% \\
    2010  & 26\%  &       & 31.05 & 11.60\% &       & 33.89 & 11.60\% &       & 31.35 & 11.60\% &       & 22.47 & 11.60\% \\
    2010H1 & 32\%  &       & 1.17  & 16.30\% &       & 1.16  & 16.30\% &       & 3.45  & 16.30\% &       & 2.14  & 16.30\% \\
    2010H2 & 15\%  &       & 1.05  & 4.52\% &       & 0.95  & 4.52\% &       & 47.99 & 4.52\% &       & 3.14  & 4.52\% \\
\bottomrule
    \end{tabular}}%
  \label{tab:MiningTBLPrelaxnoBound}%
\end{table}%

\begin{table}[H]
  \centering
  \caption{Mining Instances With Bound Tightening ({\it Source-Based}: LP Relaxations)}
  \scalebox{0.7}{
    \begin{tabular}{cccccccccccccc}
    \toprule
    \multirow{2}[4]{*}{Instance} & $D$-Gap   &       & \multicolumn{2}{c}{$\mathcal{F}_1^\mathcal{S}$} &       & \multicolumn{2}{c}{$\mathcal{F}_2^\mathcal{S}$} &       & \multicolumn{2}{c}{$\mathcal{F}_3^\mathcal{S}$} &       & \multicolumn{2}{c}{$\mathcal{F}_4^\mathcal{S}$} \\
\cmidrule{4-5}\cmidrule{7-8}\cmidrule{10-11}\cmidrule{13-14}          & (Boland et al.) &       & Time & \% $D$-Gap &       & Time & \% $D$-Gap &       & Time & \% $D$-Gap &       & Time & \% $D$-Gap \\
    \midrule
    2009H2 & 37\%  &       & 0.97  & 5.50\% &       & 0.95  & 4.98\% &       & 1.37  & 4.33\% &       & 2.42  & 4.31\% \\
    2009Q3 & 29\%  &       & 0.09  & 2.03\% &       & 0.08  & 1.68\% &       & 0.17  & 1.37\% &       & 0.22  & 1.36\% \\
    2009Q4 & 42\%  &       & 0.19  & 19.28\% &       & 0.12  & 22.06\% &       & 0.31  & 14.28\% &       & 0.50  & 14.12\% \\
    2010  & 26\%  &       & 46.08 & 9.94\% &       & 58.56 & 9.46\% &       & 49.49 & 8.44\% &       & 28.67 & 8.41\% \\
    2010H1 & 32\%  &       & 1.81  & 14.11\% &       & 1.75  & 14.54\% &       & 3.45  & 13.22\% &       & 3.91  & 13.17\% \\
    2010H2 & 15\%  &       & 1.59  & 3.72\% &       & 1.47  & 2.13\% &       & 5.89  & 1.80\% &       & 3.91  & 1.80\% \\
\bottomrule
    \end{tabular}}%
  \label{tab:MiningSBLPrelaxBound}%
\end{table}%

\begin{table}[H]
  \centering
  \caption{Mining Instances With Bound Tightening ({\it Terminal-Based}: LP Relaxations)}
  \scalebox{0.7}{
    \begin{tabular}{cccccccccccccc}
    \toprule
    \multirow{2}[4]{*}{Instance} & $D$-Gap   &       & \multicolumn{2}{c}{$\mathcal{F}_1^\mathcal{T}$} &       & \multicolumn{2}{c}{$\mathcal{F}_2^\mathcal{T}$} &       & \multicolumn{2}{c}{$\mathcal{F}_3^\mathcal{T}$} &       & \multicolumn{2}{c}{$\mathcal{F}_4^\mathcal{T}$} \\
\cmidrule{4-5}\cmidrule{7-8}\cmidrule{10-11}\cmidrule{13-14}          & (Boland et al.) &       & Time & \% $D$-Gap &       & Time & \% $D$-Gap &       & Time & \% $D$-Gap &       & Time & \% $D$-Gap \\
    \midrule
    2009H2 & 37\%  &       & 0.69  & 3.01\% &       & 0.72  & 7.69\% &       & 1.00  & 2.74\% &       & 1.83  & 2.68\% \\
    2009Q3 & 29\%  &       & 0.06  & 1.23\% &       & 0.06  & 4.02\% &       & 0.08  & 1.14\% &       & 0.19  & 1.12\% \\
    2009Q4 & 42\%  &       & 0.12  & 10.61\% &       & 0.09  & 26.31\% &       & 0.22  & 8.17\% &       & 0.44  & 8.01\% \\
    2010  & 26\%  &       & 30.65 & 7.20\% &       & 39.10 & 11.13\% &       & 31.23 & 6.83\% &       & 23.67 & 6.68\% \\
    2010H1 & 32\%  &       & 1.63  & 9.32\% &       & 1.27  & 15.61\% &       & 2.75  & 8.69\% &       & 4.03  & 8.59\% \\
    2010H2 & 15\%  &       & 1.44  & 2.18\% &       & 1.14  & 4.06\% &       & 2.53  & 2.15\% &       & 4.55  & 2.08\% \\
\bottomrule
    \end{tabular}}%
  \label{tab:MiningTBLPrelaxBound}%
\end{table}%

\subsubsection{Discretization Relaxations}

\begin{table}[H]
  \centering
  \caption{Mining Instances Without Bound Tightening (MIP Relaxations: $H=3$)}
  \scalebox{0.7}{
    \begin{tabular}{cccccccccccc}
    \toprule
    \multirow{2}[4]{*}{Instance} & \multicolumn{2}{c}{$\mathcal{M}_1^\mathcal{S}(H)$} &       & \multicolumn{2}{c}{$\mathcal{M}_2^\mathcal{S}(H)$} &       & \multicolumn{2}{c}{$\mathcal{M}_1^\mathcal{T}(H)$} &       & \multicolumn{2}{c}{$\mathcal{M}_2^\mathcal{T}(H)$} \\
\cmidrule{2-3}\cmidrule{5-6}\cmidrule{8-9}\cmidrule{11-12}          & Time & \% $D$-Gap &       & Time & \% $D$-Gap &       & Time & \% $D$-Gap &       & Time & \% $D$-Gap \\
    \midrule
    2009H2 & 3068.20 & 1.40\% &       & 3601.18 & 2.94\% &       & 3600.41 & 3.08\% &       & 701.86 & 2.08\% \\
    2009Q3 & 3.55  & 0.61\% &       & 14.35 & 1.09\% &       & 21.46 & 0.63\% &       & 6.88  & 0.60\% \\
    2009Q4 & 910.54 & 5.40\% &       & 471.86 & 13.73\% &       & 3600.13 & 11.49\% &       & 433.70 & 5.66\% \\
    2010  & 3600.44 & 8.64\% &       & 3601.90 & 9.39\% &       & 3600.42 & 8.98\% &       & 3600.30 & 10.29\% \\
    2010H1 & 3600.51 & 5.20\% &       & 3600.48 & 12.81\% &       & 3600.40 & 10.51\% &       & 3600.19 & 10.80\% \\
    2010H2 & 1526.73 & 0.84\% &       & 1409.23 & 1.22\% &       & 3311.77 & 2.50\% &       & 3461.43 & 0.95\% \\
\bottomrule
    \end{tabular}}%
  \label{tab:MiningMIPRelaxnoBound}%
\end{table}%

\begin{table}[H]
  \centering
  \caption{Mining Instances With Bound Tightening (MIP Relaxations: $H=3$)}
  \scalebox{0.7}{
    \begin{tabular}{cccccccccccc}
    \toprule
    \multirow{2}[4]{*}{Instance} & \multicolumn{2}{c}{$\mathcal{M}_1^\mathcal{S}(H)$} &       & \multicolumn{2}{c}{$\mathcal{M}_2^\mathcal{S}(H)$} &       & \multicolumn{2}{c}{$\mathcal{M}_1^\mathcal{T}(H)$} &       & \multicolumn{2}{c}{$\mathcal{M}_2^\mathcal{T}(H)$} \\
\cmidrule{2-3}\cmidrule{5-6}\cmidrule{8-9}\cmidrule{11-12}          & Time & \% $D$-Gap &       & Time & \% $D$-Gap &       & Time & \% $D$-Gap &       & Time & \% $D$-Gap \\
    \midrule
    2009H2 & 3600.86 & 0.84\% &       & 3601.17 & 1.97\% &       & 3601.22 & 0.54\% &       & 955.23 & 1.74\% \\
    2009Q3 & 6.47  & 0.22\% &       & 8.14  & 0.44\% &       & 4.62  & 0.13\% &       & 12.23 & 0.52\% \\
    2009Q4 & 654.33 & 4.06\% &       & 3600.70 & 6.71\% &       & 211.13 & 1.59\% &       & 1337.54 & 4.13\% \\
    2010  & 3600.29 & 7.81\% &       & 3601.81 & 8.41\% &       & 3600.36 & 5.90\% &       & 3600.37 & 10.14\% \\
    2010H1 & 3600.30 & 4.29\% &       & 3600.64 & 8.56\% &       & 3600.58 & 3.73\% &       & 3600.26 & 11.46\% \\
    2010H2 & 1586.15 & 0.45\% &       & 1935.09 & 1.02\% &       & 1299.25 & 0.62\% &       & 3600.32 & 1.26\% \\
\bottomrule
    \end{tabular}}%
  \label{tab:MiningMIPRelaxBound}%
\end{table}%

\subsubsection{Addition of Valid Inequalities to LP Relaxations}

\begin{table}[H]
  \centering
  \caption{Mining Instances With Bound Tightening ({\it Source-based}: LP Relaxations+$\mathcal{V}_{ab}^{\mathcal{S}}$)}
  \scalebox{0.7}{
    \begin{tabular}{cccccccccccccc}
    \toprule
    \multirow{2}[4]{*}{Instance} & Gap   &       & \multicolumn{2}{c}{{$\mathcal{F}_1^\mathcal{S}$+$\mathcal{V}_{ab}^{\mathcal{S}}$}} &       & \multicolumn{2}{c}{{$\mathcal{F}_2^\mathcal{S}$+$\mathcal{V}_{ab}^{\mathcal{S}}$}} &       & \multicolumn{2}{c}{{$\mathcal{F}_3^\mathcal{S}$+$\mathcal{V}_{ab}^{\mathcal{S}}$}} &       & \multicolumn{2}{c}{{$\mathcal{F}_4^\mathcal{S}$+$\mathcal{V}_{ab}^{\mathcal{S}}$}} \\
\cmidrule{4-5}\cmidrule{7-8}\cmidrule{10-11}\cmidrule{13-14}          & (Boland et al.) &       & Time & \% $D$-Gap &       & Time & \% $D$-Gap &       & Time & \% $D$-Gap &       & Time & \% $D$-Gap \\
    \midrule
    2009H2 & 37\%  &       & 2.39  & 5.25\% &       & 2.20  & 4.77\% &       & 2.42  & 4.31\% &       & 3.38  & 4.29\% \\
    2009Q3 & 29\%  &       & 0.22  & 1.74\% &       & 0.22  & 1.64\% &       & 0.19  & 1.34\% &       & 0.28  & 1.33\% \\
    2009Q4 & 42\%  &       & 0.42  & 17.09\% &       & 0.41  & 15.78\% &       & 0.44  & 14.23\% &       & 0.64  & 14.08\% \\
    2010  & 26\%  &       & 24.00 & 9.75\% &       & 25.12 & 9.32\% &       & 28.83 & 8.44\% &       & 45.98 & 8.41\% \\
    2010H1 & 32\%  &       & 3.58  & 13.95\% &       & 3.70  & 14.33\% &       & 3.95  & 13.22\% &       & 6.23  & 13.17\% \\
    2010H2 & 15\%  &       & 3.39  & 3.13\% &       & 3.27  & 2.11\% &       & 3.25  & 1.80\% &       & 5.12  & 1.80\% \\
\bottomrule
    \end{tabular}}%
  \label{tab:MiningSBLP+6a6bBound}%
\end{table}%

\begin{table}[H]
  \centering
  \caption{Mining Instances With Bound Tightening ({\it Terminal-based}: LP Relaxations+$\mathcal{V}_{ab}^{\mathcal{T}}$)}
  \scalebox{0.7}{
    \begin{tabular}{cccccccccccrcc}
    \toprule
    \multirow{2}[4]{*}{Instance} & Gap   &       & \multicolumn{2}{c}{$\mathcal{F}_1^\mathcal{T}$+$\mathcal{V}_{ab}^{\mathcal{T}}$} &       & \multicolumn{2}{c}{$\mathcal{F}_2^\mathcal{T}$+$\mathcal{V}_{ab}^{\mathcal{T}}$} &       & \multicolumn{2}{c}{$\mathcal{F}_3^\mathcal{T}$+$\mathcal{V}_{ab}^{\mathcal{T}}$} &       & \multicolumn{2}{c}{$\mathcal{F}_4^\mathcal{T}$+$\mathcal{V}_{ab}^{\mathcal{T}}$} \\
\cmidrule{4-5}\cmidrule{7-8}\cmidrule{10-11}\cmidrule{13-14}          & (Boland et al.) &       & Time & \% $D$-Gap &       & Time & \% $D$-Gap &       & Time & \% $D$-Gap &       & Time & \% $D$-Gap \\
    \midrule
    2009H2 & 37\%  &       & 1.97  & 2.99\% &       & 2.25  & 5.23\% &       & 2.42  & 2.74\% &       & 2.80  & 2.68\% \\
    2009Q3 & 29\%  &       & 0.20  & 1.21\% &       & 0.27  & 2.52\% &       & 0.16  & 1.14\% &       & 0.30  & 1.12\% \\
    2009Q4 & 42\%  &       & 0.37  & 8.60\% &       & 0.44  & 11.27\% &       & 0.37  & 8.15\% &       & 0.77  & 7.98\% \\
    2010  & 26\%  &       & 39.79 & 7.20\% &       & 29.98 & 10.43\% &       & 47.73 & 6.83\% &       & 82.30 & 6.68\% \\
    2010H1 & 32\%  &       & 3.87  & 9.32\% &       & 3.33  & 14.77\% &       & 5.60  & 8.69\% &       & 5.06  & 8.59\% \\
    2010H2 & 15\%  &       & 3.77  & 2.18\% &       & 4.03  & 3.10\% &       & 4.05  & 2.15\% &       & 5.56  & 2.08\% \\
\bottomrule
    \end{tabular}}%
  \label{tab:MiningTBLP+6a6bBound}%
\end{table}%

\begin{table}[H]
  \centering
  \caption{Mining Instances With Bound Tightening ({\it Source-Based}: LP Relaxations+$\mathcal{V}_{ac}^{\mathcal{S}}$)}
  \scalebox{0.7}{
    \begin{tabular}{cccccccccccccc}
    \toprule
    \multirow{2}[4]{*}{Instance} & Gap   &       & \multicolumn{2}{c}{$\mathcal{F}_1^\mathcal{S}$+$\mathcal{V}_{ac}^{\mathcal{S}}$} &       & \multicolumn{2}{c}{$\mathcal{F}_2^\mathcal{S}$+$\mathcal{V}_{ac}^{\mathcal{S}}$} &       & \multicolumn{2}{c}{$\mathcal{F}_3^\mathcal{S}$+$\mathcal{V}_{ac}^{\mathcal{S}}$} &       & \multicolumn{2}{c}{$\mathcal{F}_4^\mathcal{S}$+$\mathcal{V}_{ac}^{\mathcal{S}}$} \\
\cmidrule{4-5}\cmidrule{7-8}\cmidrule{10-11}\cmidrule{13-14}          & (Boland et al.) &       & Time & \% $D$-Gap &       & Time & \% $D$-Gap &       & Time & \% $D$-Gap &       & Time & \% $D$-Gap \\
    \midrule
    2009H2 & 37\%  &       & 1.00  & 5.49\% &       & 1.25  & 4.98\% &       & 1.87  & 4.33\% &       & 2.22  & 4.31\% \\
    2009Q3 & 29\%  &       & 0.08  & 2.02\% &       & 0.08  & 1.67\% &       & 0.11  & 1.37\% &       & 0.16  & 1.36\% \\
    2009Q4 & 42\%  &       & 0.19  & 19.12\% &       & 0.12  & 22.02\% &       & 0.30  & 14.23\% &       & 0.47  & 14.08\% \\
    2010  & 26\%  &       & 25.44 & 9.94\% &       & 28.22 & 9.46\% &       & 21.62 & 8.44\% &       & 57.93 & 8.41\% \\
    2010H1 & 32\%  &       & 2.42  & 14.11\% &       & 3.24  & 14.54\% &       & 3.64  & 13.22\% &       & 4.56  & 13.17\% \\
    2010H2 & 15\%  &       & 3.34  & 3.50\% &       & 2.81  & 2.13\% &       & 3.41  & 1.80\% &       & 5.27  & 1.69\% \\
\bottomrule
    \end{tabular}}%
  \label{tab:MiningSBLP+6a6cBound}%
\end{table}%

\begin{table}[H]
  \centering
  \caption{Mining Instances With Bound Tightening ({\it Terminal-Based}: LP Relaxations+$\mathcal{V}_{ac}^{\mathcal{T}}$)}
  \scalebox{0.7}{
    \begin{tabular}{cccccccccccrcc}
    \toprule
    \multirow{2}[4]{*}{Instance} & Gap   &       & \multicolumn{2}{c}{$\mathcal{F}_1^\mathcal{T}$+$\mathcal{V}_{ac}^{\mathcal{T}}$} &       & \multicolumn{2}{c}{$\mathcal{F}_2^\mathcal{T}$+$\mathcal{V}_{ac}^{\mathcal{T}}$} &       & \multicolumn{2}{c}{$\mathcal{F}_3^\mathcal{T}$+$\mathcal{V}_{ac}^{\mathcal{T}}$} &       & \multicolumn{2}{c}{$\mathcal{F}_4^\mathcal{T}$+$\mathcal{V}_{ac}^{\mathcal{T}}$} \\
\cmidrule{4-5}\cmidrule{7-8}\cmidrule{10-11}\cmidrule{13-14}          & (Boland et al.) &       & Time & \% $D$-Gap &       & Time & \% $D$-Gap &       & Time & \% $D$-Gap &       & Time & \% $D$-Gap \\
    \midrule
    2009H2 & 37\%  &       & 0.75  & 3.01\% &       & 0.69  & 7.69\% &       & 0.99  & 2.74\% &       & 1.66  & 2.68\% \\
    2009Q3 & 29\%  &       & 0.05  & 1.23\% &       & 0.05  & 4.02\% &       & 0.06  & 1.14\% &       & 0.17  & 1.12\% \\
    2009Q4 & 42\%  &       & 0.12  & 10.61\% &       & 0.11  & 26.31\% &       & 0.22  & 8.17\% &       & 0.44  & 8.01\% \\
    2010  & 26\%  &       & 19.78 & 7.20\% &       & 24.98 & 11.13\% &       & 16.40 & 6.83\% &       & 34.86 & 6.68\% \\
    2010H1 & 32\%  &       & 2.58  & 9.32\% &       & 3.28  & 15.61\% &       & 1.98  & 8.69\% &       & 5.55  & 8.59\% \\
    2010H2 & 15\%  &       & 2.55  & 2.18\% &       & 2.58  & 4.06\% &       & 2.08  & 2.15\% &       & 4.37  & 2.08\% \\
\bottomrule
    \end{tabular}}%
  \label{tab:MiningTBLP+6a6cBound}%
\end{table}%

\end{document}